\documentclass[reqno,twoside,12pt,a4paper]{amsart}
\usepackage{bbm}
\usepackage{amsfonts}
\usepackage{graphicx,amssymb,mathrsfs,amsmath}
\usepackage{amsthm}
\usepackage{hyperref}
\usepackage{soul,color}
\usepackage{geometry}
\usepackage{enumerate}
\usepackage{xcolor}
\usepackage{ulem}

\newcommand{\norm}[1]{{\left\nuert\kern-0.25ex\left\vert\kern-0.25ex\left\vert #1\right\vert\kern-0.25ex\right\vert\kern-0.25ex\right\vert}}

\newcommand{\N}{\mathbb{N}}
\newcommand{\R}{\mathbb{R}}
\newcommand{\T}{\mathbb{T}}

\newcommand{\eps}{\mathbb{\varepsilon}}
\newcommand{\al}{\alpha}
\newcommand{\dive}{\mathrm{div}}

\newcommand{\pt}{\partial_t}
\newcommand{\pa}{\partial}

\newcommand{\D}{\nabla}

\newcommand{\vp}{\varphi}

\newcommand{\calf}{\mathcal{F}}

\newcommand{\cubic}{C\mathscr{E}_T^{1/2}\mathscr{D}(t)}
\newcommand{\Ni}{\mathcal{N}_i}
\newcommand{\Ne}{\mathcal{N}_e}

\newtheorem{theorem}{Theorem}[section]

\newtheorem{lemma}{Lemma}[section]
\newtheorem{proposition}{Proposition}[section]

\numberwithin{equation}{section}
\allowdisplaybreaks

\begin{document}
	\title{\bf Global convergence and error estimates in infinity-ion-mass limits for Bipolar Euler-Poisson system}
	
	\author{Yachun Li$^{1}$, Shihao Wang$^{2,3}$ and Liang Zhao$^{4,*}$}
	\date{}
	\maketitle \markboth{Convergence rates in infinity-ion-mass limits for Euler-Poisson system}{Y. Li, S. Wang and L. Zhao}
	
	\vspace{-3mm}
	
	\begin{center}
		{\small 
			$^1$School of Mathematical Sciences, CMA-Shanghai, MOE-LSC, and SHL-MAC  \\Shanghai Jiao Tong University
\\Shanghai 200240, P. R. China
\\[2mm]			
	$^2$School of Mathematical Sciences, Shanghai Jiao Tong University\\
			Shanghai 200240, P.R. China\\[2mm]
			$^3$Department of Mathematics, University of Wisconsin-Madison\\
			Madison 53706, United States\\[2mm]
			$^4$Mathematical Modelling \& Data Analytics Center,
			Oxford Suzhou Centre for Advanced Research\\
			Suzhou 215123, P.R. China
		
		}
	\end{center}

	\footnotetext{$*$ Corresponding author \\
		E-mail address: liang.zhao@oxford-oscar.cn}
	
	\vspace{1cm}
	
	\begin{center}
		\begin{minipage}{15cm}
			\small
			{\bf Abstract.} This paper is concerned with the global-in-time convergence from bipolar Euler–Poisson system (BEP) to unipolar one (UEP) through the infinity-ion-mass limit by letting the ratio of the mass of ion $m_i$ over that of electron $m_e$ goes to infinity. 
			The global convergence of the limit is obtained for smooth solutions sufficiently close to constant equilibrium states. Furthermore, by applying the stream function method and taking advantage of the anti-symmetric structure of the error system, one obtains the corresponding global-in-time error estimates between smooth solutions of (BEP) and (UEP). It is worth mentioning that due to the strong coupling through the Poisson equation in bipolar system, stream functions for ions and electrons equations should be constructed separately based on asymptotic expansions of solutions, which is very different from the case of unipolar system. 
		\end{minipage}
	\end{center}
	
	\vspace{7mm}
	
	
	\vspace{5mm}
	
	\noindent {\bf AMS Subject Classification (2020)~:}  35B25, 35L45, 35L60, 35Q35, 35Q60
	
	\vspace{5mm}
	
	\section{Introduction}
 As the limiting system under the non-relativistic limit for the Euler-Maxwell system (See \cite{Peng2007,Peng2017,Wasiolek2016}), Euler-Poisson system plays an important role in describing the motions of charged fluids (ions and electrons) in semi-conductors or plasma when the effect of the magnetic field is weak. The periodic problem for the bipolar Euler-Poisson system (BEP) is of the form \cite{Ali2003,Chen1984,Markowich1990},
\begin{equation}\label{start}
\begin{cases}
\pt\rho_{\nu}+\mathrm{div}\left(\rho_{\nu}u_{\nu}\right)=0,
\\
m_{\nu}\pt\left(\rho_{\nu}u_{\nu}\right)+m_{\nu}\mathrm{div}\left(\rho_{\nu}u_{\nu}\otimes u_{\nu}\right)+\nabla p_\nu(\rho_\nu)=q_\nu \rho_\nu\nabla \vp-m_\nu\rho_\nu u_\nu,\\
\triangle \vp=\rho_i-\rho_e,\\
t=0: (\rho_\nu,u_\nu) = (\rho_{\nu,0},u_{\nu,0}),
\end{cases}
\end{equation}
for $t>0$ and $x=(x_1,\cdots,x_d)\in\T^d$, where $\T^d$ is a torus in $\R^d$. Here for $\nu=e,i$ with $e$ standing for electrons and $i$ ions, $ \rho_\nu$ stand for the particle density and $u_\nu$ the average velocity of ions and electrons respectively, $ \vp $  is the scaled electric potential. The pressure functions $ p_\nu(\rho_\nu) $
are supposed to be smooth and strictly increasing for all $ \rho_\nu > 0 $.  The parameters $m_\nu$ stand for the mass of a single particle and $(q_e,q_i)=(-1,1)$ stand for the charge of particles. 

The (BEP) system is symmetrizable hyperbolic for $\rho_{\nu}>0$ if regarding $\D \vp$ as a zeroth order term of densities $\rho_\nu$ by the Poisson equation, hence the existence of  local smooth solution is guaranteed by classical hyperbolic theories \cite{Kato1975,Lax1973,Majda1984}. For smooth initial data, Al{\`\i} \cite{Ali2003} proved the global existence of smooth solutions,
which are sufficiently close to the equilibrium state, to equations \eqref{start}, see also Li-Yang \cite{Li2012} and Tong-Tan-Xu \cite{Tong2020} for the decay rate problem. In Wu-Li \cite{Wu2016} and Wu-Wang \cite{Wu2014}, the authors studied the pointwise estimates of solutions. See also Peng-Xu \cite{P2013} and Xu-Kawashima \cite{Xu2015} when initial data is in critical Besov spaces.

On the scale of electrons, (BEP) can be simplified based on the fact that the velocity of electrons is much greater than that of ions. Consequently, ions are often regarded as non-moving and becoming a uniform background density for simplicity. In fact, by assuming 
\[
\rho_i=b(x),\qquad u_i=0,
\] 
(BEP) becomes the following unipolar Euler-Poisson model for electrons,
\begin{equation}\label{unipolar}
\begin{cases}
\pt\bar{\rho}_{e}+\mathrm{div}\left(\bar{\rho}_{e}\bar{u}_{e}\right)=0,
\\
m_{e}\pt\left(\bar{\rho}_{e}\bar{u}_{e}\right)+m_{e}\mathrm{div}\left(\bar{\rho}_{e}\bar{u}_{e}\otimes \bar{u}_{e}\right)+\nabla p_e(\bar{\rho}_e)=-\bar{\rho}_e\nabla \bar{\vp}-m_e \bar{\rho}_e \bar{u}_e,\\
\triangle \bar{\vp}=b(x)-\bar{\rho}_e,
\end{cases}
\end{equation}
which has been widely studied, see for instance \cite{Ali2000, Guo1998, Hsiao2003, Peng2015}. However, the passage from (BEP) to (UEP) is only based on physical observations and lack of  rigorous mathematical proof. 

In this paper, we will give the rigorous global-in-time convergence analysis from (BEP) to (UEP) in the perspective of mass. Based on the large difference between the mass of electrons and ions, one can assume $m_e/m_i \rightarrow 0$. In order to derive the (UEP) system \eqref{unipolar}, we let $m_e=1$ be fixed and $1/m_i\to 0$, known as the infinity-ion-mass limit, which is another interpretation of the zero-electron-mass limit ($m_i$ fixed and $m_e\to 0$, see for instance \cite{Ali2011,Giuseppe2010,Goudon1999,Peng2015,Xu2011}). The infinity-ion-mass limit was introduced in Xi-Zhao \cite {Xi2020}, where the authors obtained the local convergence of this limit for (BEP) without damping mechanism for velocities. We also refer readers to Zhao \cite{Zhao2021a} for this limit applied for bipolar Euler-Maxwell system. 

The global convergence of the infinity-ion-mass limit can be obtained by establishing uniform energy estimates with respect to the small parameter $m_e/m_i$. When it comes to the global-in-time error estimates between smooth solutions to (BEP) and (UEP), the energy estimates for error variables do not close as the structure has changed. For this difficulty, a stream function method should be applied to get a $L^2$ dissipative energy estimate for $\rho_\nu-\bar{\rho}_\nu$. For a conservative equation
\begin{equation*}
	\pt z + \dive w=0,
\end{equation*}
we call $\psi$ a stream function associated to this equation if $\psi$ satisfies
\[
\left\{
\begin{aligned}
&\pt \psi=w+K,\\
&\dive \psi=-z,
\end{aligned}
\right.
\] 
with some divergence-free term $K$. The solution of the above system can be established by letting
\[
\psi(t,x)=\psi(0,x)+\int_{0}^{t} (w(\tau,x)+K(\tau,x))d\tau, \qquad \text{with}\,\,\,\,\dive \psi(0,x)=-z(0,x).
\]
In our case, the conservative equation can be chosen as the difference between the mass equations of \eqref{start} and \eqref{unipolar} with $z=\rho_e-\bar{\rho}_e$. Then the inner product of the stream function $\psi$ with the difference between momentum equations of \eqref{start} and \eqref{unipolar} can admit a dissipative energy for $\rho_e-\bar{\rho}_e$ by noticing the strict monotonicity of the pressure function. This method has been successfully used for many fluid dynamical models. We refer readers to the isothermal Euler system with damping in Junca-Rascle \cite{Junca2002}, the M1 model in Goudon-Lin \cite{GOUDON2013579} and the Euler-Maxwell and Euler-Poisson system \cite{LI2021185}, see also \cite{Li2021, Zhao2021a}. However, these results are all relaxation type and are for unipolar systems. Some essential difficulties arise for bipolar models under the infinity-ion-mass limit:

\begin{itemize}
	\item The pressure term in the momentum equation for ions vanishes after the passage of the limit, which makes the classical stream function method invalid for the ion equations. Besides, different from relaxation type limits, the dissipative estimates for $\rho_e-\bar{\rho}_e$ are coupled with $u_e-\bar{u}_e$, which can not be obtained only based on stream function methods. 
	\item  For bipolar models, the equations for electrons can not be treated independently as the strong coupling effect of $\rho_e$ and $\rho_i$ results in the failure of the natural stream function candidate $\D(\vp-\bar{\vp})$.
\end{itemize}

In order to handle these difficulties, different treatments are applied for electron and ion equations. For ion equations,  due to the vanishing of the dissipative structure for $\rho_i-\bar{\rho}_i$, we construct the stream function for ion based on some auxiliary equations. For electron equations,  the global $L^2-$estimates is obtained and energy estimates are closed based on the anti-symmetric structure (see for instance \cite{Guo2011,Peng2015a}) and the stream function technique. To our best knowledge, this seems to be the first global error estimate result for (BEP) concerning bipolar systems under limits other than relaxation limits.

In the following, we let $m_e=1$ and define the small parameter $\eps=m_i^{-1/2}$. Thus the infinity-ion-mass limit means letting $\eps\rightarrow 0$. We denote $s>\frac{d}{2}+1$ an integer and $ C>0$  a generic constant independent of the small parameter $ \eps $ and any time.  For a multi-index $\alpha=(\alpha_1,\cdots,\alpha_d)\in\mathbb{N}^d$,
we denote
\[  \pa_x^\alpha=\dfrac{\pa^{|\alpha|}}{\pa x_1^{\alpha_1}\cdots\pa x_d^{\alpha_d}}
\quad {\text {with}} \quad |\alpha|=\alpha_1+\cdots+\alpha_d.  \]
For simplicity, we denote by $\|\cdot\|$, $\|\cdot\|_\infty$ and $\|\cdot\|_l$
the usual norms of $L^2\stackrel{def}{=}L^2(\T^d)$,
$L^\infty\stackrel{def}{=}L^\infty(\T^d)$ and $H^l\stackrel{def}{=}H^l(\T^d)$
for all integers $l\ge 1$, respectively. The inner product in $L^2(\T^d)$ is denoted as $\left<\cdot, \cdot\right>$. We denote $\nu=i,e$, where $i$ stands for ions and $e$ electrons. 

Now we begin to show the main results of the current paper. Theorem \ref{Thm2.2} concerns the global-in-time uniform estimates of smooth solutions to (BEP) near the constant equilibrium states in classical Sobolev spaces $H^s(\T^d)$.
\begin{theorem}\label{Thm2.2} (Uniform estimates with respect to $\eps$)
   Let $s>\frac{d}{2}+1$ be an integer. There exist positive constants $C_1$ and $\delta$ independent of $\eps$ such that for all $\eps\in(0,1]$, if
  \begin{equation*}\label{delta}
   \sum_{\nu=i,e}\|\rho_{\nu,0}^\eps-1\|_s+\dfrac{1}{\eps}\|\rho_{i,0}^\eps-1\|_{s-1}+\|u_{e,0}^\eps\|_s+\dfrac{1}{\eps}\|u_{i,0}^\eps\|_s\leq \delta,
  \end{equation*}
then for all $t>0$, the periodic problem \eqref{start} admits a unique global solution $(\rho_\nu^\eps, u_\nu^\eps, \varphi ^\eps)$ satisfying
\begin{eqnarray}\label{finestim}
&\quad& \sum_{\nu=i,e}\|\rho_\nu^\eps(t)-1\|_s^2+\dfrac{1}{\eps^2}\|\rho_i^\eps(t)-1\|_{s-1}^2+\dfrac{1}{\eps^2}\|u_i^\eps(t)\|_s^2
+\|u_e^\eps(t)\|_s^2+\|\D\vp^\eps(t)\|_s^2\nonumber\\
&&+\int_0^t\left(\sum_{\nu=i,e}\|\nabla \rho_\nu^\eps(\tau)\|_{s-1}^2+\dfrac{1}{\eps^2}\|u_i^\eps(\tau)\|_s^2+\|u_e^\eps(\tau)\|_s^2
+\|\Delta \vp^\eps(\tau)\|_{s-1}^2\right)d\tau\nonumber\\
&\leq&C_1\left(\|\rho_{i,0}^\eps-1\|_s^2+\|\rho_{e,0}^\eps-1\|_s^2+\dfrac{1}{\eps^2}\|\rho_{i,0}^\eps-1\|_{s-1}^2+\|u_{e,0}^\eps\|_s^2+\dfrac{1}{\eps^2}\|u_{i,0}^\eps\|_s^2\right).
  \end{eqnarray}
\end{theorem}

Combining with some compactness theories, see for instance \cite{Simon1987}, the following Theorem \ref{Thm2.3} concerning the global-in-time convergence of inifinity-ion-mass limit is obtained.
\begin{theorem}\label{Thm2.3}(Global convergence) Let $(\rho_\nu^\eps, u_\nu^\eps, \vp^\eps)$ be the unique global smooth solution to (BEP) obtained in Theorem \ref{Thm2.2}. If, as $\eps\rightarrow 0$,
\begin{equation*}
(\rho_{i,0}^\eps, \rho_{e,0}^\eps, u_{e,0}^\eps, u_{i,0}^\eps)\rightharpoonup (\bar{\rho}_{i,0}, \bar{\rho}_{e,0}, \bar{u}_{e,0},0), \,\,\,\text{weakly in } \,\, H^s,
\end{equation*}
then there exist functions $(\bar{\rho}_i,\bar{\rho}_e, \bar{u}_e, \bar{\varphi})\in L^\infty(\R^+;H^s)$, such that, as $\eps\rightarrow0$, up to subsequences,
\begin{eqnarray}
u_i^\eps&\rightarrow& 0 \quad \text{strongly in } \,\,\, L^\infty([0,T];H^s), \quad \forall\,T>0,\nonumber\\
\label{con2}(\rho_i^\eps,\rho_e^\eps, u_e^\eps, \D\vp^\eps)&\stackrel{*}\rightharpoonup& (\bar{\rho}_i, \bar{\rho}_e, \bar{u}_e, \D\bar{\vp}), \quad \text{weakly-* in } \,\,\, L^\infty(\R^+;H^s),
\end{eqnarray}
where $\bar{\rho}_i=\bar{\rho}_i(x)$ depends only on the space variable $x$, and $(\bar{\rho}_e, \bar{u}_e, \bar{\vp})$ is the unique global smooth solution to the unipolar system \eqref{unipolar} when $m_e=1$.
\end{theorem}

Based on the above two theorems, one obtains the following Theorem \ref{thm4.1} concerning the global-in-time error estimates between (BEP) and (UEP) in a three-dimensional torus $\T^3$.
\begin{theorem}\label{thm4.1}
(Global convergence rate) Let $d=3$ and $s \geq 3$ be integers. Let the conditions in Theorems \ref{Thm2.2}-\ref{Thm2.3} hold. Let $(\rho_i^\eps,\rho_e^\eps, u_i^\eps, u_e^\eps, \varphi^\eps)$ be the global smooth solution to (BEP) obtained in Theorem \ref{Thm2.2} and $(\bar{\rho}_e, \bar{u}_e, \bar{\varphi})$ be the global solution to (UEP) obtained in Theorem \ref{Thm2.3}. Further assume there exists a constant $C_2>0$ independent of $\eps$ and any time such that
   \begin{equation}\label{cond2}
    \|\rho^\eps_{e,0}-\bar{\rho}_{e,0}\|_{s-1}+\dfrac{1}{\eps}\|\rho^\eps_{i,0}-\bar{\rho}_i(x)\|_{s-1}+\|u^\eps_{e,0}-\bar{u}_{e,0}\|_s + \dfrac{1}{\eps}\|u^\eps_{i,0}\|_s \leq C_2\eps,
  \end{equation}
  then for all $\eps \in (0,1]$, there exists a positive constant $C_3$ independent of $\eps$ and any time, such that the following estimates hold
  \begin{eqnarray*}
    && \sup_{\tau \in \mathbb{R}^+} (\|(u_e^\eps-\bar{u}_e)(\tau)\|^2_{s-1} + \|\rho_i^\eps(\tau)-\bar{\rho}_i\|^2_{s-1} + \|(\rho_e^\eps-\bar{\rho}_e)(\tau)\|^2_{s-1}+ \| \nabla (\varphi^\eps-  \bar{\varphi})(\tau)\|^2_{s-1})\\
    &&+ \int_{0}^{\infty}(\|(u^\eps_e-\bar{u}_e)(\tau)\|^2_{s-1} + \|\rho^\eps_i(\tau)-\bar{\rho}_i\|^2_{s-1} + \|(\rho^\eps_e-\bar{\rho}_e)(\tau)\|^2_{s-1}+ \| \nabla (\varphi^\eps-  {\varphi})(\tau)\|^2_{s-1}) d\tau \\
    &\leq& C_3 \eps^2,
  \end{eqnarray*}
  and
  \begin{equation*}
      \sup_{\tau \in \mathbb{R}^+} \|u^\eps_i(\tau)\|^2_{s-1}+  \int_{0}^{\infty}\|u^\eps_i(\tau)\|^2_{s-1} d\tau \leq C_3 \eps^4.
  \end{equation*}
\end{theorem}

This paper is organized as follows. Section 2 states preliminaries and reformulations. In Section 3, uniform estimates of (BEP) with respect to $\eps$ are established, based on which the global-in-time convergence towards (UEP) is obtained. Sections 4 focuses on the corresponding convergence rate.

\section{Preliminaries and reformulations}

In this section, we show some preliminaries and reformulations of system (BEP). The formal derivation of the limiting equations are also listed. 
\subsection{Reformulations  and local existence of smooth solutions} We first write (BEP) into a symmetric hyperbolic system. By writing the original momentum equation in \eqref{start} into convective form when $\displaystyle\inf_{x\in\T^3}\rho_{\nu}^\eps(x)>0$, system \eqref{start} becomes
\begin{equation}\label{main}
\begin{cases}
\pt \rho_i^\eps+\dive(\rho_i^\eps u_i^\eps)=0,\\
\pt u_i^\eps+(u_i^\eps\cdot\nabla)u_i^\eps+\eps^2\nabla h_i(\rho_i^\eps)=\eps^2\D\vp^\eps-u_i^\eps,\\
\pt \rho_e^\eps +\dive(\rho_e^\eps u_e^\eps)=0,\\
\pt u_e^\eps+(u_e^\eps\cdot\nabla)u_e^\eps+\nabla h_e(\rho_e^\eps)=-\D\vp^\eps-u_e^\eps, \\
\Delta\vp^\eps=\rho_i^\eps-\rho_e^\eps, 
\end{cases}
\end{equation}
where $h_\nu$ are the enthalpy functions are defined by
\[
h_\nu^\prime(\rho)=\dfrac{p_\nu^\prime(\rho)}{\rho},
\]
which are strictly increasing since $p_\nu'(\rho)>0$. The initial conditions of system \eqref{main} reads
\begin{equation*}
    t=0: \rho_{\nu}^\eps(0,x)=\rho_{\nu,0}^\eps(x),\   \
    u_{\nu}^\eps(0,x)=u_{\nu,0}^\eps(x).
\end{equation*}
Meanwhile, the following restriction is imposed to make $ \vp^\eps $ uniquely determined,
\begin{equation*}
	\int_{\T^d}\vp^\eps(x)dx=0.
\end{equation*}
the initial condition of $\vp^\eps$ is defined as follows to avoid appearance of the initial layer:
\[
\Delta \vp_0^\eps=\rho^\eps_{i,0}-\rho^\eps_{e,0}.
\]
System \eqref{main} admits an equilibrium state
\begin{eqnarray*}
	(\rho^\eps_e, \rho^\eps_i, u^\eps_e, u^\eps_i, \D\vp^\eps)=(1,1,0,0,0).
\end{eqnarray*}
Note that we can take $\D\vp^\eps$ as a zeroth order term as the $H^s$ norm of $\D\vp^\eps$ can be controlled based on
\[
\|\D\vp^\eps\|_s \leq C\|\D^2\vp^\eps\|_s \leq C\|\dive \D\vp^\eps\|_s\leq C (\|\rho^\eps_{i,0}-1\|_s +\|\rho^\eps_{e,0}-1\|_s),
\]
where the Poincar\'e inequality is used. Now we drop superscript $\eps$ for simplicity and denote
\[
N_\nu= \rho_\nu-1, \quad U_\nu=\left(\begin{matrix}
N_\nu\\
u_\nu\end{matrix}\right).
\]
For $1\leq j\leq d$ and $u_\nu=(u_\nu^1, \cdots, u_\nu^d)$, we denote
\begin{eqnarray*}
A_i^j(\rho_i, u_i)=\left(\begin{matrix}
     u_i^j & \rho_i \xi_j^\top\\
      \eps^2 h_i^\prime(\rho_i) \xi_j & u_i^j \mbox {\bf I}_d\\
     \end{matrix}\right),\qquad
A_e^j(\rho_e, u_e)=\left(\begin{matrix}
      u_e^j & \rho_e \xi_j^\top  \\
      h_e^\prime(\rho_e) \xi_j &   u_e^j \mbox{\bf I}_d \\
    \end{matrix}\right),
\end{eqnarray*}
and
\begin{eqnarray*}
Q_i(u_i,\vp)=\left(\begin{matrix}
                               0 \\
                               \eps^{2}\D\vp-u_i \\
                             \end{matrix}
\right),\qquad
Q_e(u_e,\vp)=\left(\begin{matrix}
                               0 \\
                              -\D\vp-u_e
                              \end{matrix}\right),
\end{eqnarray*}
where ${\bf{I}}_d$ is the $d\times d$ unit matrix and $\{\xi_k\}_{k=1}^d$ is the canonical basis of $\R^d$.  Finally, system \eqref{main} is rewritten into the following form
\begin{equation}\label{startsymmglobal}
  \pt U_\nu+\sum_{j=1}^d A_\nu^j(\rho_\nu, u_\nu)\pa_{x_j}U_\nu=Q_\nu(u_\nu, \D\vp),
\end{equation}
with initial conditions
\begin{eqnarray}\label{startsymmglobalinitial}
t=0: \quad (\rho_\nu,u_\nu, \varphi)=(\rho_{\nu,0}^\eps, u_{\nu,0}^\eps, \varphi_0^\eps), \quad x\in\T^d.
\end{eqnarray}
System \eqref{startsymmglobal}-\eqref{startsymmglobalinitial} is symmetrizable hyperbolic when $\rho_\nu>0$. Indeed, if we define the symmetrizers as
\begin{eqnarray*}
A_i^0(\rho_i)=\left(\begin{matrix}
                  h_i^\prime(\rho_i)&0 \\
                   0 & \eps^{-2} \rho_i \mbox{\bf I}_d\\
                   \end{matrix}\right),\qquad
A_e^0(\rho_e)=\left(\begin{matrix}
                  h_e^\prime(\rho_e)&0 \\
                   0 &  \rho_e \mbox {\bf I}_d\\
                   \end{matrix}\right),
\end{eqnarray*}
which are symmetric and positive definite for $\rho_\nu>0$, then the matrices $\tilde{A}_\nu^j$  defined by
\begin{eqnarray*}
\tilde{A}_i^j(\rho_i, u_i)&=&A_i^0(\rho_i) A_i^j(\rho_i, u_i) =\left(\begin{matrix}
 h_i^\prime(\rho_i)u_i^j &  p_i^\prime(\rho_i) \xi_j^\top\\
 p_i^\prime(\rho_i) \xi_j& \dfrac{1}{\eps^2}\rho_i u_i^j \mbox {\bf I}_d\\
\end{matrix}
\right),\\
\tilde{A}_e^j(\rho_e, u_e)&=&A_i^0(\rho_e) A_e^j(\rho_e, u_e) =\left(\begin{matrix}
 h_e^\prime(\rho_e)u_e^j &  p_e^\prime(\rho_e) \xi_j^\top\\
 p_e^\prime(\rho_e) \xi_j&  \rho_e u_e^j \mbox {\bf I}_d
\end{matrix}\right),
\end{eqnarray*}
 are symmetric. By the theory of Lax \cite{Lax1973} and Kato \cite{Kato1975} for the symmetrizable hyperbolic system, one has
\begin{proposition}\label{localexistence}(Local existence of smooth solutions)
Let  $ s \geq \dfrac{d}{2}+1$  be an integer and the initial data $ \left(\rho_{\nu,0}^\eps,\,u_{\nu,0}^\eps\right)\in H^s $  with  $ \rho_{\nu,0}^\eps \geq 2\underline{\rho} $  for some given positive constant  $ \underline{\rho}>0$ independent of  $ \eps $. 
Then there exists  $ T^\eps>0 $  such that the periodic problem \eqref{main} has a unique smooth solution  $ \left(\rho^\eps_\nu,\,u^\eps_\nu,\,\vp^\eps \right) $  defined on the time interval $ [0,T^\eps] $, satisfying  $ \rho_\nu^\eps \geq \underline{\rho} $  and
\begin{eqnarray*}
\left(\rho_\nu^\eps, u_\nu^\eps, \D\vp^\eps\right)&\in &C\left(\left[0,\,T^\eps\right];\,H^s\right)\cap C^1\left(\left[0,\,T^\eps\right];\,H^{s-1}\right).
\end{eqnarray*}
\end{proposition}

\subsection{Formal Derivation of the Limiting System.} 
We now formally derive the limiting system. The formal limits of $(\rho_i^\eps,\rho_e^\eps,u_i^\eps,u_e^\eps,\varphi^\eps)$ are denoted by $(\bar{\rho}_i,\bar{\rho}_e,\bar{u}_i,\bar{u}_e,\nabla \bar{\varphi})$. Then as $\eps\to 0$, the formal limiting system reads
\begin{equation}\label{formallimit}
\begin{cases}
\pt \bar{\rho}_i+\dive(\bar{\rho}_i \bar{u}_e)=0,\\
\pt \bar{u}_i+(\bar{u}_i\cdot\nabla)\bar{u}_i+\bar{u}_i=0,\\
\pt \bar{\rho}_e +\dive(\bar{\rho}_e \bar{u}_e)=0,\\
\pt \bar{u}_e+(\bar{u}_e\cdot\nabla)\bar{u}_e+\nabla h_e(\bar{\rho}_e)=-\D\bar{\vp}-\bar{u}_e, \\
\Delta\bar{\vp}=\bar{\rho}_i-\bar{\rho}_e. 
\end{cases}
\end{equation}
Let $\bar{v}_i=e^t \bar{u}_i$, the second equation in \eqref{formallimit} becomes
\begin{equation*}
    \pt \bar{v}_i+(\bar{u}_i\cdot\nabla)\bar{v}_i=0.
\end{equation*}
Consider the characteristic path $x(y,0;t)$ starting from $y$ and satisfying
\[\frac{d}{dt}x(y,0;t)=\bar{u}_i(x(y,0;t),t),\]
then one obtains
\[
\frac{d}{dt}\bar{v}_i(x(y,0;t),t)=0,
\] namely $\bar{v}_i$ is a constant along this path. This yields
\[y=x-\int_0^t \bar{u}_i(x,\tau) d\tau=x-\int_0^t e^{-\tau} \bar{v}_i(x,\tau) d\tau = x+\bar{u}_i(x,t)- e^{-t} \bar{u}_i(x,t),
\]
which gives $\bar{u}_i(x,t)=e^{-t}u^\eps_{i,0}(x+(1- e^{-t}) u^\eps_{i}(x,t))$. Thus by imposing $u_{i,0}^\eps\to 0$, one has $\bar{u}_i \to 0$. Hence $\partial_t \bar{\rho}_i=0$ by the first equation in \eqref{formallimit}, namely $\bar{\rho}_i$ only depends only on the spatial variables. Thus, the equations for electrons in  \eqref{formallimit} are decoupled from the bipolar system and become self-contained, which are of the form
\begin{equation}\label{efinal}
\begin{cases}
\pt\bar{\rho}_e +\dive(\bar{\rho}_e \bar{u}_e)=0,\\
\pt \bar{u}_e +(\bar{u}_e\cdot\nabla)\bar{u}_e+\nabla h_e(\bar{\rho}_e)=-\D\bar{\vp}-\bar{u}_e,\\
\Delta \bar{\vp}=\bar{\rho}_i(x)-\bar{\rho}_e,
\end{cases}
\end{equation}
with initial conditions
\begin{equation}\label{efinalinitial}
t=0: (\bar{\rho}_e, \bar{u}_e)=(\bar{\rho}_{e,0},\bar{u}_{e,0}).
\end{equation}

\section{Uniform Estimates and Global convergence}
In this section, we denote by $s> \frac{d}{2}+1$ an integer and $ C $  a generic positive constant independent of $ \eps $ and any time. We first prove Theorem \ref{Thm2.2}. In what follows, we will drop the superscript $\eps$ for simplicity and assume that the conditions in Theorem \ref{Thm2.2} hold. We denote
\[
 W=\left(
N_i,\eps^{-1}u_i,N_e,u_e,\D\vp\right)^\top.
\]
Without loss of generality, we suppose $T>0$ and $(\rho_\nu,u_\nu,\vp)$ is the local solution to \eqref{main} on $[0,T]$. We introduce the total energy and the dissipative energy as follows
\begin{eqnarray*}
\mathscr{E}(t)&=&\sum_{\nu=i,e}\|N_\nu\|_s^2+\dfrac{1}{\eps^2}\|N_i(t)\|_{s-1}^2+\dfrac{1}{\eps^2}\|u_i(t)\|_s^2
+\|u_e(t)\|_s^2+\|\D \vp(t)\|_s^2,\\
\mathscr{D}(t)&=&\sum_{\nu=i,e}\|\nabla N_\nu(t)\|_{s-1}^2+\dfrac{1}{\eps^2}\|u_i(t)\|_s^2+\|u_e(t)\|_{s}^2.
\end{eqnarray*}
Moreover, we set
\begin{equation*}\label{totalenergy}
\mathscr{E}_T=\sup_{0\leq t\leq T}\mathscr{E}(t),
\end{equation*}
which is assumed to be uniformly sufficiently small with respect to $T$ and $\eps$. Besides, because of the smallness of $\mathscr{E}_T$, it is reasonable to assume
\begin{equation}\label{lowerboundrho}
 \dfrac{1}{2}\leq \rho_\nu \leq \dfrac{3}{2} \quad  \text{and}\quad  h_\nu^\prime(\rho_\nu)\geq h_1,
\end{equation}
where $h_1$ is a positive constant independent of the small parameter $\eps$ and any time. One first obtains
\begin{lemma}($L^2-$estimate) It holds
\begin{equation}\label{L2estim}
\|W(t)\|^2+\int_0^t\left(\dfrac{1}{\eps^2}\|u_i(\tau)\|^2+\|u_e(\tau)\|^2\right)d\tau\leq C\|W(0)\|^2.
\end{equation}
\end{lemma}
\begin{proof}The entropy and the corresponding entropy flux for the bipolar  Euler system in \eqref{main} are
\[
\begin{cases}
  \eta_0(\rho_\nu,u_\nu)=\dfrac{1}{2}\rho_e|u_e|^2+ H_e(\rho_e)+ \dfrac{1}{2\eps^2}\rho_i|u_i|^2+H_i(\rho_i),\\[2mm]
  \psi_0(\rho_\nu,u_\nu)=\dfrac{1}{2}\rho_e|u_e|^2u_e+ \rho_eh_e(\rho_e)u_e+\dfrac{1}{2\eps^2}\rho_i|u_i|^2u_i+ \rho_ih_i(\rho_i)u_i,
 \end{cases}
\]
in which $H_\nu^\prime(\rho)=h_\nu(\rho)$. Then we have
\begin{equation*}
 \pt \eta_0(\rho_\nu,u_\nu) +\dive \psi_0(\rho_\nu,u_\nu)+ \rho_e |u_e|^2+\dfrac{1}{\eps^2}\rho_i |u_i|^2 =-\D\vp(\rho_e u_e-\rho_i u_i).
\end{equation*}
Combining the Poisson equation and the conservative equations we have
\begin{eqnarray*}
-\D\vp(\rho_e u_e-\rho_i u_i)&=&-\dive(\vp(\rho_e u_e-\rho_i u_i))+\vp\dive(\rho_e u_e-\rho_i u_i)\nonumber\\
&=&-\dive(\vp(\rho_e u_e-\rho_i u_i))-\vp\pt(\rho_e-\rho_i)\nonumber\\
&=&-\dive(\vp(\rho_e u_e-\rho_i u_i))+\vp\pt\Delta\vp\nonumber\\
&=&-\dive(\vp(\rho_e u_e-\rho_i u_i))+\dive(\vp\pt\D\vp)-\dfrac{1}{2}\dfrac{d}{dt}|\D\vp|^2.\nonumber
\end{eqnarray*}
Thus we have
\begin{eqnarray}\label{wshed}
 &&\pt \left(\eta_0(\rho_\nu,u_\nu)+\dfrac{1}{2}|\D\vp|^2\right)\nonumber\\
 &&+\dive \left(\psi_0(\rho_\nu,u_\nu)+\vp(\rho_eu_e-\rho_iu_i-\pt\D\vp)\right)+ \rho_e |u_e|^2+\dfrac{1}{\eps^2}\rho_i |u_i|^2 =0.
\end{eqnarray}
In addition, by the Taylor's formula, we obtain
\[
H_\nu(\rho_\nu)=H_\nu(1)+h_\nu(1)N_\nu+ \dfrac{1}{2}h_\nu^\prime(\rho_\nu^*)
N_\nu^2,
\]
where $\rho_\nu^*$ is between $1$ and $\rho_\nu$. Since $\pt N_\nu=-\dive(\rho_\nu u_\nu)$, we have
\[
\pt H_\nu(\rho_\nu) =-h_\nu(1)\dive(\rho_\nu u_\nu)+ \dfrac{1}{2}\pt\left(h_\nu^\prime(\rho_\nu^*) N_\nu^2\right).
\]
Substituting the above equation to \eqref{wshed} and integrating the resulting equation over $\T^d$ and $[0,t]$, also noticing \eqref{lowerboundrho} yields \eqref{L2estim}. 
\end{proof}

\begin{lemma}\label{lemmahoestim} (Higher order estimates) It holds
\begin{eqnarray}\label{hoestim}
\|W(t)\|_s^2+\int_0^t\left(\|u_e(\tau)\|_s^2+\dfrac{1}{\eps^2}\|u_i(\tau)\|_s^2\right)d\tau\leq C\|W(0)\|_s^2+C\int_0^t\mathscr{E}_T^{1/2}\mathscr{D}(\tau)d\tau.
\end{eqnarray}
\end{lemma}

\begin{proof} For a multi-index $\alpha\in\N^d$ with $1\leq |\alpha|\leq s$, applying $\pa_x^\al$ to both sides of \eqref{startsymmglobal}, and making the inner product of the resulting equation with $2A_\nu^0(\rho_\nu) \pa_x^\al U_\nu$ in $L^2(\T^d)$, we obtain
\begin{eqnarray}\label{starthigher}
\dfrac{d}{dt}\left<\pa_x^\al U_\nu, A_\nu^0(\rho_\nu)\pa_x^\al U_\nu\right> &=&\left<\dive A_\nu \pa_x^\al U_\nu, \pa_x^\al U_\nu\right> +2\left<J_\nu^\al, A_\nu^0\pa_x^\al U_\nu\right> +2\left<\pa_x^\al Q_\nu, A_\nu^0(\rho_\nu)\pa_x^\al U_\nu \right>\nonumber\\
&=&I_\nu^1+I_\nu^2+I_\nu^3,
\end{eqnarray}
with the natural correspondence of $I_\nu^1$, $I_\nu^2$ and $I_\nu^3$, and
in which
\begin{eqnarray*}
\dive A_\nu &=&\pt A_\nu^0(\rho_\nu)+\sum_{j=1}^d \pa_{x_j}\tilde{A}_\nu^j(\rho_\nu, u_\nu),\nonumber\\
J_\nu^\al&=&\sum_{j=1}^d(A_\nu^j(\rho_\nu,u_\nu)\pa_x^\al \pa_{x_j}U_\nu-\pa_x^\al (A_\nu^j(\rho_\nu,u_\nu)\pa_{x_j}U_\nu)).
\end{eqnarray*}

In what follows, we will estimate $I_\nu^1$, $I_\nu^2$ and $I_\nu^3$ term by term. First, since $\pt \rho_\nu=-\dive(\rho_\nu u_\nu)$, we have
\[
\|\pt \rho_\nu\|_{\infty}\leq C\|u_\nu\|_s.
\]
Since $\eps< 1$, it follows that
\begin{eqnarray*}
\left<\pt A_\nu^0(\rho_\nu)\pa_x^\al U_\nu, \pa_x^\al U_\nu\right> \leq\cubic.
\end{eqnarray*}
Similarly, for $1\leq j\leq d$,
\begin{eqnarray*}
\left<\pa_{x_j}\tilde{A}_\nu^j (\rho_\nu, u_\nu)\pa_x^\al U_\nu, \pa_x^\al U_\nu\right>\leq \cubic.
\end{eqnarray*}
Therefore
\begin{equation}\label{I1}
|I_\nu^1|\leq \cubic.
\end{equation}

For $I_\nu^2$, since $I_\nu^2=2\sum_{j=1}^d I_{\nu,j}^2$, with
\begin{eqnarray*}
I_{i,j}^2&=&\left<h^\prime_i(\rho_i)\left(\pa_x^\al(u_i^j\pa_{x_j}N_i)-u_i^j\pa_x^\al\pa_{x_j}N_i\right),\pa_x^\al N_i\right>\\
  &\quad&+\left<h^\prime_i(\rho_i)\left(\pa_x^\al(\rho_i\pa_{x_j}u_i^j)-\rho_i\pa_x^\al\pa_{x_j}u_i^j\right),\pa_x^\al N_i\right>\\
  &\quad&+\left<\rho_i\left(\pa_x^\al(h_i^\prime(\rho_i)\pa_{x_j}N_i)-h_i^\prime(\rho_i)\pa_x^\al\pa_{x_j}N_i\right),\pa_x^\al u_i^j\right>\\
    &\quad&+\eps^{-2}\left<\rho_i\left(\pa_x^\al(u_i^j\pa_{x_j}u_i)-u_i^j\pa_x^\al\pa_{x_j}u_i\right),\pa_x^\al u_i\right>,
\end{eqnarray*}
and
\begin{eqnarray*}
I_{e,j}^2&=&\left<h^\prime_e(\rho_e)\left(\pa_x^\al(u_e^j\pa_{x_j}N_e)-u_e^j\pa_x^\al\pa_{x_j}N_e\right),\pa_x^\al N_e\right>\\
  &\quad&+\left<h^\prime_e(\rho_e)\left(\pa_x^\al(\rho_e\pa_{x_j}u_e^j)-\rho_e\pa_x^\al\pa_{x_j}u_e^j\right),\pa_x^\al N_e\right>\\
  &\quad&+\left<\rho_e\left(\pa_x^\al(h_e^\prime(\rho_e)\pa_{x_j}N_e)-h_e^\prime(\rho_e)\pa_x^\al\pa_{x_j}N_e\right),\pa_x^\al u_e^j\right>\\
    &\quad&+\left<\rho_e\left(\pa_x^\al(u_e^j\pa_{x_j}u_e)-u_e^j\pa_x^\al\pa_{x_j}u_e\right),\pa_x^\al u_e\right>.
\end{eqnarray*}
Noticing $\nabla h_\nu^\prime(\rho_\nu)=h_\nu^{\prime\prime}(\rho_\nu)\nabla N_\nu$, by the Moser-type inequalities in Lemma 2.1, we have
\begin{eqnarray}\label{I2}
 |I_\nu^2|\leq \cubic.
\end{eqnarray}

For $I_\nu^3$, by using the Moser-type inequalities,  we obtain
\begin{eqnarray*}
I_i^3&=&2\left<\pa_x^\al \D\vp, \rho_i\pa_x^\al u_i\right>-\dfrac{2}{\eps^2}\left<\rho_i \pa_x^\al u_i, \pa_x^\al u_i\right>\nonumber\\
&\leq&2\left<\pa_x^\al \D\vp, \rho_i\pa_x^\al u_i\right>-\dfrac{1}{\eps^2}\|\pa_x^\al u_i\|^2,
\end{eqnarray*}
similarly
\begin{eqnarray*}
I_e^3&\leq&-2\left<\pa_x^\al \D\vp, \rho_e\pa_x^\al u_e\right>-\|\pa_x^\al u_e\|^2.
\end{eqnarray*}
Substituting \eqref{I1}, \eqref{I2} and the above two estimates into \eqref{starthigher}, and adding the resulting equation for $\nu=i,e$ yield
\begin{eqnarray}\label{firstpart}
&&\sum_{\nu=i,e}\dfrac{d}{dt}\left<\pa_x^\al U_\nu, A_\nu^0(\rho_\nu) \pa_x^\al U_\nu\right> +\dfrac{1}{\eps^2}\|\pa_x^\al u_i\|^2 +\|\pa_x^\al u_e\|^2\nonumber\\
&\leq&\cubic+2\left<\pa_x^\al \D\vp, \rho_i\pa_x^\al u_i-\rho_e \pa_x^\al u_e\right>.
\end{eqnarray}

Now we estimate $2\left<\pa_x^\al \D\vp, \rho_i\pa_x^\al u_i-\rho_e \pa_x^\al u_e\right>$. Since
\begin{eqnarray}
\left<\pa_x^\al \D\vp, \rho_i\pa_x^\al u_i-\rho_e \pa_x^\al u_e\right>&=&\left<\pa_x^\al\D\vp, \pa_x^\al(\rho_iu_i-\rho_eu_e)\right>-\sum_{\nu=i,e}\left<\pa_x^\al \D\vp, \pa_x^\al(\rho_\nu u_\nu)-\rho_\nu\pa_x^\al u_\nu\right>,\nonumber
\end{eqnarray}
in which 
\begin{eqnarray}
\left<\pa_x^\al\D\vp, \pa_x^\al(\rho_iu_i-\rho_eu_e)\right>&=&-\left<\pa_x^\al\vp, \dive\pa_x^\al(\rho_i u_i-\rho_eu_e)\right>\nonumber\\
&=&\left<\pa_x^\al\vp, \pt\pa_x^\al(\rho_i-\rho_e)\right>\nonumber\\
&=&\left<\pa_x^\al\vp, \pt\Delta\pa_x^\al\vp\right>\nonumber\\
&=&-\dfrac{1}{2}\dfrac{d}{dt}\|\pa_x^\al\D\vp\|^2,\nonumber
\end{eqnarray}
and
\[
\left|\left<\pa_x^\al \D\vp, \pa_x^\al(\rho_e u_e)-\rho_e\pa_x^\al u_e\right>\right|\leq C\|\pa_x^\al\D\vp\|\|\D N_e\|_{s-1}\|u_e\|_{s-1}\leq \cubic,
\]
\[
\left|\left<\pa_x^\al \D\vp, \pa_x^\al(\rho_i u_i)-\rho_i\pa_x^\al u_i\right>\right|\leq C\|\pa_x^\al\D\vp\|\|\D N_i\|_{s-1}\|\eps^{-1} u_i\|_{s-1}\leq \cubic.
\]
Substituting these into \eqref{firstpart}, adding the resulting equation for all $1\leq |\al|\leq s$ and integrating over $[0,t]$, we have
\begin{eqnarray*}
&&\sum_{1\leq|\al|\leq s}\left( \sum_{\nu=i,e} \left<\pa_x^\al U_\nu, A_\nu^0(\rho_\nu)\pa_x^\al U_\nu\right>+ \|\pa_x^\al\D\vp\|^2\right)+\int_0^t\left(\|u_e(\tau)\|_s^2+\dfrac{1}{\eps^{2}} \|u_i(\tau)\|_s^2\right)d\tau\nonumber\\
&\leq& C\|W(0)\|_s^2+C\int_0^t\mathscr{E}_T^{1/2}\mathscr{D}(\tau)d\tau.
\end{eqnarray*}
Noticing that there exists constants $c_a>0$ such that
\[
\left<\pa_x^\al U_\nu, A_\nu^0(\rho_\nu)\pa_x^\al U_\nu\right>+ \|\pa_x^\al\D\vp\|^2\geq c_a\|\pa_x^\al W\|^2,
\]
further combining \eqref{L2estim} yield \eqref{hoestim}. 
\end{proof}

\begin{lemma}\label{lemmadisestim}(Dissipation of $\nabla N_\nu$) It holds
\begin{eqnarray}\label{disestim}
&\,&\sum_{|\beta|\leq s-1}\dfrac{d}{dt}\left(\dfrac{1}{\eps^2}\left( \|\pa_x^\beta N_i\|^2+2\left<\pa_x^\beta (\rho_i u_i), \pa_x^\beta \nabla N_i \right>\right)+\|\pa_x^\beta N_e\|^2+2\left<\pa_x^\beta (\rho_e u_e), \pa_x^\beta \nabla N_e \right>\right)\nonumber\\
&\,&+\dfrac{h_1}{2}\|\nabla N_i\|_{s-1}^2+\dfrac{h_1}{2}\|\nabla N_e\|_{s-1}^2+\dfrac{1}{2}\|\Delta \vp\|_{s-1}^2\nonumber\\
&\leq &C\left\|\dfrac{u_i}{\eps}\right\|_s^2+C\|u_e\|_s^2+\cubic.
\end{eqnarray}

\end{lemma}
\begin{proof}For a multi-index $\beta\in\N^d$ with $|\beta|\leq s-1$, multiplying $\eps^{-2}\rho_i$ to both sides of the momentum equation of ions in \eqref{start}, and applying $\pa_x^\beta$ to the resulting equation, we have
\begin{eqnarray*}
p_i^\prime(\rho_i)\pa_x^\beta \nabla N_i-\pa_x^\beta (\rho_i\D\vp) &=&-\eps^{-2}\pt\pa_x^\beta(\rho_i u_i)-\eps^{-2}\pa_x^\beta(\dive(\rho_i u_i\otimes u_i))-\eps^{-2}\pa_x^\beta (\rho_iu_i)\nonumber\\
&\,&+\left(p_i^\prime(\rho_i)\pa_x^\beta \nabla N_i-\pa_x^\beta(p_i^\prime(\rho_i)\nabla N_i)\right).
\end{eqnarray*}
Taking the inner product of the above equation with $\pa_x^\beta\nabla N_i$ in $L^2(\T^d)$ yields
\begin{eqnarray}\label{dissstart}
&\,&\left<p_i^\prime(\rho_i)\pa_x^\beta \nabla N_i, \pa_x^\beta\nabla N_i\right>
-\left<\pa_x^\beta (\rho_i\D\vp) , \pa_x^\beta\nabla N_i\right>\nonumber\\
&=&-\left<\eps^{-2}\pt\pa_x^\beta(\rho_i u_i), \pa_x^\beta\nabla N_i\right>-\left<\eps^{-2}\pa_x^\beta(\dive(\rho_i u_i\otimes u_i)), \pa_x^\beta\nabla N_i\right>
-\left<\eps^{-2}\pa_x^\beta (\rho_iu_i), \pa_x^\beta\nabla N_i\right>\nonumber\\
&\,&+\left<\left(p_i^\prime(\rho_i)\pa_x^\beta \nabla N_i
-\pa_x^\beta(p_i^\prime(\rho_i)\nabla N_i)\right), \pa_x^\beta\nabla N_i\right>.
\end{eqnarray}
Now we treat each term in the above equation. First, the estimate for the term containing $\pa_x^\beta(\rho_i \D\vp)$ on the left hand side requires a little more calculations. Let us first remark that
\begin{eqnarray*}
\left<\pa_x^\beta (\rho_i \D\vp), \pa_x^\beta \nabla N_i\right>
&=&\left<\rho_i \pa_x^\beta \D\vp, \pa_x^\beta \nabla N_i\right>+\left<\pa_x^\beta(\rho_i \D\vp)-\rho_i \pa_x^\beta \D\vp, \pa_x^\beta \nabla N_i\right>,
\end{eqnarray*}
in which by the Moser-type inequalities, we have
\[
\left|\left<\pa_x^\beta(\rho_i \D\vp)-\rho_i \pa_x^\beta \D\vp, \pa_x^\beta \nabla N_i\right>\right|\leq C\|\D\vp\|_{s-1}\|\nabla N_i\|_{s-1}^2\leq \cubic,
\]
and
\begin{eqnarray*}
\left<\rho_i \pa_x^\beta \D\vp, \pa_x^\beta \nabla N_i\right>
&=&\dfrac{1}{2}\left<\pa_x^\beta \D\vp, \pa_x^\beta(\nabla (\rho_i)^2)\right>-\left<\pa_x^\beta \D\vp, \pa_x^\beta(\rho_i\nabla \rho_i)-\rho_i\pa_x^\beta \nabla \rho_i\right>\nonumber\\
&\leq& \dfrac{1}{2}\left<\pa_x^\beta \D\vp, \pa_x^\beta(\nabla (\rho_i)^2)\right>+\cubic.
\nonumber
\end{eqnarray*}
Hence,
\begin{equation*}
\left<\pa_x^\beta (\rho_i \D\vp), \pa_x^\beta \nabla N_i\right>\leq \dfrac{1}{2}\left<\pa_x^\beta \D\vp, \pa_x^\beta(\nabla (\rho_i)^2)\right>+\cubic.
\end{equation*}
Moreover, noticing \eqref{lowerboundrho}, we have
\[
\left<p_i^\prime(\rho_i)\pa_x^\beta \nabla N_i, \pa_x^\beta \nabla N_i\right> \geq \dfrac{h_1}{2}\|\pa_x^\beta \nabla N_i\|^2.
\]

Let us move on to the right hand side of \eqref{dissstart}. An integration by parts gives
\begin{eqnarray*}
\eps^{-2} \left<\pt\pa_x^\beta (\rho_iu_i), \pa_x^\beta \nabla N_i\right> &=& \eps^{-2}\dfrac{d}{dt}\left<\pa_x^\beta (\rho_i u_i), \pa_x^\beta \nabla N_i \right>-\eps^{-2}\left<\pa_x^\beta \dive (\rho_i u_i), \pa_x^\beta \dive(\rho_iu_i)\right>\nonumber\\
&=&\eps^{-2}\dfrac{d}{dt}\left<\pa_x^\beta (\rho_i u_i), \pa_x^\beta \nabla N_i \right> -\dfrac{1}{\eps^2}\left\|\pa_x^\beta\dive(\rho_i u_i)\right\|^2.
\end{eqnarray*}
By applying the Young's inequality, the Cauchy-Schwarz inequality and the Moser-type inequalities, it is obvious that
\begin{eqnarray*}
\eps^{-2}|\left<\pa_x^\beta(\dive(\rho_i u_i\otimes u_i)),\pa_x^\beta \nabla N_i\right>|&\leq& C\|\nabla N_i\|_{s-1}\left\|\dfrac{u_i}{\eps}\right\|_s^2\leq \cubic ,\nonumber\\
|\left<p_i^\prime(\rho_i)\pa_x^\beta \nabla N_i-\pa_x^\beta(p_i^\prime(\rho_i)\nabla N_i),\pa_x^\beta \nabla N_i\right>|&\leq& C \|\nabla N_i\|_{s-1}^3\leq \cubic.
\end{eqnarray*}
Next, using the mass equation of ions in \eqref{main}, we have
\[
\eps^{-2}\left<\pa_x^\beta (\rho_i u_i),\pa_x^\beta \nabla N_i\right> =-\eps^{-2}\left<\pa_x^\beta \dive(\rho_i u_i), \pa_x^\beta N_i \right>=\eps^{-2}\left<\pt\pa_x^\beta N_i, \pa_x^\beta N_i\right>=\dfrac{1}{2\eps^2}\dfrac{d}{dt}\|\pa_x^\beta N_i\|^2.
\]

Thus, combining all these estimates, we have
\begin{eqnarray}\label{idiss}
&\,& \dfrac{1}{\eps^2}\dfrac{d}{dt}\left( \|\pa_x^\beta N_i\|^2+2\left<\pa_x^\beta (\rho_i u_i), \pa_x^\beta \nabla N_i \right>\right)+ h_1\|\pa_x^\beta\nabla N_i\|^2-\left<\pa_x^\beta \D\vp, \pa_x^\beta(\nabla (\rho_i)^2)\right>\nonumber\\
&\leq& C\left\|\dfrac{u_i}{\eps}\right\|_s^2+\cubic.
\end{eqnarray}
Similarly, applying the same procedure as above to the momentum equation for electrons, we obtain the similar estimate for electrons:
\begin{eqnarray}\label{ediss}
&\,& \dfrac{d}{dt}\left( \|\pa_x^\beta N_e\|^2+2\left<\pa_x^\beta (\rho_e u_e), \pa_x^\beta \nabla N_e \right>\right)+ h_1\|\pa_x^\beta\nabla N_e\|^2+\left<\pa_x^\beta \D\vp, \pa_x^\beta(\nabla (\rho_e)^2)\right>\nonumber\\
&\leq& C\left\|{u}_e\right\|_s^2+C\mathscr{E}_T^{1/2}\mathscr{D}(t).
\end{eqnarray}
Now it remains to estimate the term $\left<\pa_x^\beta \D\vp, \pa_x^\beta \nabla ((\rho_i)^2-(\rho_e)^2)\right>$. Since
\[
\pa_x^\beta \Delta \vp= \pa_x^\beta \rho_i-\pa_x^\beta \rho_e,
\]
then
\begin{eqnarray*}
\left<\pa_x^\beta \D\vp, \pa_x^\beta \nabla ((\rho_i)^2-(\rho_e)^2)\right>
&=&-\left<\pa_x^\beta \Delta \vp, \pa_x^\beta ((\rho_i-\rho_e)(\rho_i+\rho_e))\right>\nonumber\\
&=&-\left<(\rho_i+\rho_e)\pa_x^\beta \Delta \vp, \pa_x^\beta \Delta \vp\right>\nonumber\\
&\,&-\left<\pa_x^\beta \Delta \vp, \pa_x^\beta ((\rho_i-\rho_e)(\rho_i+\rho_e))-(\rho_i+\rho_e)\pa_x^\beta(\rho_i-\rho_e)\right>,
\end{eqnarray*}
in which noticing \eqref{lowerboundrho}, we have
\[
\left<(\rho_i+\rho_e)\pa_x^\beta \Delta \vp, \pa_x^\beta \Delta \vp\right>\geq \|\pa_x^\beta \Delta \vp \|^2,
\]
and by the Cauchy-Schwarz inequality and the Moser-type inequalities, we have
\begin{eqnarray*}
&\,&|\left<\pa_x^\beta \Delta \vp, \pa_x^\beta ((\rho_i-\rho_e)(\rho_i+\rho_e))-(\rho_i+\rho_e)\pa_x^\beta(\rho_i-\rho_e)\right>|\nonumber\\
&\leq& C\|\pa_x^\beta \Delta \vp\|\|\nabla \rho_i+\nabla \rho_e\|_{s-1} \|\rho_i-\rho_e\|_{s-1}\nonumber\\
&\leq &\dfrac{1}{2}\|\pa_x^\beta \Delta \vp\|^2+\cubic.
\end{eqnarray*}
Hence, combining these estimates, we have
\[
\left<\pa_x^\beta \D\vp, \pa_x^\beta \nabla ((\rho_i)^2-(\rho_e)^2)\right>\leq -\dfrac{1}{2}\|\pa_x^\beta \Delta \vp\|^2+C\mathscr{E}_T^{1/2}\mathscr{D}(t).
\]

Adding \eqref{idiss} and \eqref{ediss}, and combining the above estimate, we have
\begin{eqnarray*}
&\,&\dfrac{1}{\eps^2}\dfrac{d}{dt}\left( \|\pa_x^\beta N_i\|^2+2\left<\pa_x^\beta (\rho_iu_i), \pa_x^\beta \nabla N_i \right>\right)+\dfrac{d}{dt}\left( \|\pa_x^\beta N_e\|^2+2\left<\pa_x^\beta (\rho_e u_e), \pa_x^\beta \nabla N_e \right>\right)\nonumber\\
&\,&+\dfrac{h_1}{2}\|\pa_x^\beta\nabla N_i\|^2+\dfrac{h_1}{2}\|\pa_x^\beta\nabla N_e\|^2+\dfrac{1}{2}\|\pa_x^\beta\Delta \vp\|^2\nonumber\\
&\leq &C\left\|\dfrac{u_i}{\eps}\right\|_s^2+C\|u_e\|_s^2+C\mathscr{E}_T^{1/2}\mathscr{D}(t).
\end{eqnarray*}
Adding the above for all $|\beta|\leq s-1$ yields \eqref{disestim}. 
\end{proof}

\begin{lemma}For $\forall\, t>0$, it holds
 \begin{equation}\label{finestimbf}
\mathscr{E}(t)+\int_0^t \mathscr{D}(\tau)d\tau\leq C\mathscr{E}(0),
\end{equation}
\end{lemma}
\begin{proof} Now let us define the following
\begin{eqnarray*}
\mathbb{E}(t)&=&\kappa\eps^{-2}\| N_i\|_{s-1}^2+\kappa\| N_e\|_{s-1}^2+\|\nabla \varphi\|^2_s\\
&\,&+ \int_{\T^d}\left(\rho_e|u_e|^2+\sum_{\nu=i,e}h^\prime_\nu(\rho_\nu^*)N_\nu^2+\dfrac{1}{\eps^2}\rho_i|u_i|^2\right)dx\\
&\,&+\sum_{\nu=i,e} \sum_{1\leq |\al| \leq s} \left\langle \pa_x^\al U_\nu, A_\nu^0 \pa_x^\al U_\nu \right\rangle\\
&\,&+\sum_{|\beta|\leq s-1}2\kappa\left(\eps^{-2}\left<\pa_x^\beta (\rho_i u_i), \pa_x^\beta \nabla N_i \right>+\left<\pa_x^\beta (\rho_e u_e), \pa_x^\beta \nabla N_e \right>\right),\nonumber
\end{eqnarray*}
and
\begin{eqnarray*}
\mathbb{D}(t)&=&\|u_e\|_s^2+\dfrac{1}{\eps^{2}} \|u_i\|_s^2+\kappa h_1\|\nabla N_i\|_{s-1}^2+\kappa h_1\|\nabla N_e\|_{s-1}^2+\kappa \|\Delta \vp\|_{s-1}^2,
\end{eqnarray*}
where $\kappa>0$ is a small constant to be determined later.  Integrating \eqref{disestim} over $[0,t]$ and using Lemma \ref{lemmahoestim}-\ref{lemmadisestim}, adding the two resulting estimates in the way \eqref{hoestim}+$\kappa$\eqref{disestim}, we have
\[
\mathbb{E}(t)+ \int_0^t\mathbb{D}(\tau)d\tau\leq \mathbb{E}(0)+C\kappa\int_0^t\left(\|u_e(\tau)\|_s^2+\dfrac{1}{\eps^{2}} \|u_i(\tau)\|_s^2d\tau\right)+ C\int_0^t\mathscr{E}_T^{1/2}\mathscr{D}(\tau)d\tau.
\]
Since  $\mathbb{D}(t)$ is equivalent to $\mathscr{D}(t)$, and $\mathscr{E}_T$ is sufficiently small, a straightforward calculation implies
that there exists a positive constant $\kappa_1>0$, such that
\begin{equation}\label{fin}
\mathbb{E}(t)+ \kappa_1\int_0^t \mathscr{D}(\tau)d\tau\leq \mathbb{E}(0),
\end{equation}
provided that $\kappa$ is chosen to be sufficiently small. Since $\mathbb{E}(t)$ is equivalent to $\mathscr{E}(t)$ as $A_\nu^0$ is symmetric and positive-definite. Thus, from \eqref{fin}, we have \eqref{finestimbf}, which also implies the uniform global existence of solution to \eqref{main} by a bootstrap argument. Finally, by the Poincar\'e inequality,
\[
\|\D\vp\|\leq C\|\D^2\vp\|\leq C\|\dive \D\vp\|,
\]
which gives the uniform estimate \eqref{finestim}.
\end{proof}

\noindent \textbf{Proof of Theorem \ref{Thm2.3}.} The uniform estimate \eqref{finestimbf} implies that
\[
u_i^\eps \longrightarrow 0, \,\,\,\text{strongly in}\,\,\,C([0,T];H^s), \quad \forall \,T>0.
\]
Besides, we obtain that the sequences $(\rho_\nu^\eps-1)_{\eps>0}$,  $(u_e^\eps)_{\eps>0}$ and $(\D\vp^\eps)_{\eps>0}$ are uniformly bounded in $L^\infty(\R^+;H^s)$. It follows that there exist functions $\bar{\rho_\nu},\bar{u}_e$ and $\bar{\varphi}$, such that as $\eps\rightarrow 0$, \eqref{con2} holds. This allows us to pass the limit in the mass and momentum equations for ions in the sense of distributions. In particular, we have
\begin{eqnarray*}
\pt \rho_i &\to& \pt \bar{\rho}_i, \\
\dive(\rho_i^\eps u_i^\eps)&\to& 0,
\end{eqnarray*}
which implies $\pt \bar{\rho}_i=0.$ Thus, $\bar{\rho}_i$ is a function that depends only on the space variable $x$.

Moreover, noticing $\eps<1$, $(\pt\rho_e^\eps)_{\eps>0}$ and $(\pt u_e^\eps)_{\eps>0}$ are uniformly bounded in  $L^\infty(\R^+; H^{s-1})$,  by a classical compactness theorem \cite{Simon1987}, $(\rho_e^\eps)_{\eps>0}$ and $(u_e^\eps)_{\eps>0}$ are relatively compact in $C([0,T];H_{loc}^{s_1})$, for all $s_1\in(0,s)$. As a consequence, as $\eps\rightarrow 0$, up to subsequences,
\begin{equation}\label{conver1}
(\rho_e^\eps, u_e^\eps)\rightarrow (\bar{\rho}_e,\bar{u}_e), \quad \text{strongly in }\,\, C([0,T];H_{loc}^{s_1}).
\end{equation}
As a result, it is sufficient for us to pass the limit in the mass and momentum equations for electrons, as well as the Poisson equations in \eqref{main} in the sense of distributions, of which the limiting system is the usual unipolar Euler-Poisson system for electrons \eqref{efinal}. \hfill $\square$

\section{Global convergence rate}
In this section, we study the global-in-time convergence rate problems. For simplicity, the superscript $\eps$ is dropped. In this section, we set $d=3$. Let $(\rho_i, u_i, \rho_e, u_e, \D\vp)$ be the solution to bipolar Euler-Poisson system \eqref{main}, and let their corresponding limit as $(\bar{\rho}_i(x),0,\bar{\rho}_e, \bar{u}_e, \D\bar{\vp})$, in which $(\bar{\rho}_e, \bar{u}_e, \D\bar{\vp})$ satisfies the unipolar Euler-Poisson system for electrons \eqref{efinal}. 

The complete proof has a total of three steps. First, we introduce $\bar{u}_i$ as the limit of $\eps^{-2}u_i$ to construct an auxiliary equation as the substitute for the missing limiting equation of ion. Then we apply the anti-symmetric structure of the error system to establish some preliminary conclusions. Finally, the stream function method is applied to establish the dissipative estimate for $(\rho_i-\bar{\rho_i})$ and $(\rho_e-\bar{\rho_e})$, thus the energy estimate is closed combining with former conclusions.

\subsection{The Limiting Equations} At the beginning, we first list some estimates and some facts about the limiting equations. First, under the conditions in Theorems \ref{Thm2.2} and \ref{thm4.1}, we have $\bar{\rho}_i(x)=1$. Indeed, 
\[
    \|\bar{\rho}_i(x)-1\|_{s-1}\ \leq \|\bar{\rho}_i(x)-\rho_{i,0}\|_{s-1}+\|\rho_{i,0}-1\|_{s-1} \leq C \eps,
\]
for all $\eps\in(0,1]$. Based on this, $(\bar{\rho}_e, \bar{u}_e, \D\bar{\vp})$  is in a neighbourhood of a constant equilibrium state. Thus the limit system \eqref{efinal} becomes
\begin{equation}\label{efinal2}
\begin{cases}
\pt\bar{\rho}_e +\dive(\bar{\rho}_e \bar{u}_e)=0,\\
\pt \bar{u}_e +(\bar{u}_e\cdot\nabla)\bar{u}_e=-\D\bar{\vp}-\bar{u}_e,\\
\Delta \bar{\vp}=1-\bar{\rho}_e,
\end{cases}
\end{equation}
with the same initial condition \eqref{efinalinitial}. It is clear that there are rich literatures on this system. Indeed, we have the following proposition.
\begin{proposition}\label{danjiep}(Estimates for equations for electrons, see \cite{Hsiao2003,Peng2015}) Let $s\geq 3$ be an integer. There exists positive constants $C$ and $\delta$ such that when 
\[
\|\bar{\rho}_{e,0}-1\|_s+\|\bar{u}_{e,0}\|_s\leq \delta,
\]
then there exists a unique global smooth solution $(\bar{\rho}_e, \bar{u}_e, \D\bar{\vp})$  to \eqref{efinal} satisfying
\begin{eqnarray}\label{eregu}
&\quad& \|\bar{\rho}_e(t)-1\|_s^2+\|\bar{u}_e(t)\|_s^2+\|\D\bar{\vp}(t)\|_s^2+\int_0^t\|\bar{\rho}_e(\tau)-1\|_s^2+\|\bar{u}_e(\tau)\|_s^2+\|\D\bar{\vp}(\tau)\|_s^2d\tau\nonumber\\
&\leq& C(\|\bar{\rho}_{e,0}-1\|_s^2+\|\bar{u}_{e,0}\|_s^2).
  \end{eqnarray}
Moreover, there exists a positive constant $q>0$, such that
\[
\|\bar{\rho}_e(t)-1\|_s^2+\|\bar{u}_e(t)\|_s^2+\|\D\bar{\vp}(t)\|_s^2\leq C(\|\bar{\rho}_{e,0}-1\|_s^2+\|\bar{u}_{e,0}\|_s^2)e^{-qt},
\]
which implies that the solution admits an exponential decay when $t\to +\infty$. 
\end{proposition}

However, for the limiting equations for ions, after taking the limits, we lose the information of the conservative equations for ions. Consequently, we use the first order profile in the asymptotic expansion to substitute the limiting equations for ions. In order to do this, we first need to establish estimates on $u_i$. Since $u_i\to 0$, the estimates on $u_i$ actually imply  the error estimates for $u_i$. 

\begin{lemma}\label{convu}(Convergence rate for $u_i$)\label{u_i} Suppose $\|u_{i,0}\|_{s-1}\leq C\eps^2$, then we have
\begin{equation}\label{finalpre2} 
    \|u_i(t)\|_{s-1}^2+\int_0^t\|u_i(\tau)\|_{s-1}^2d\tau\leq C\eps^4, \quad \forall\,t>0.
\end{equation}
\end{lemma}

\begin{proof}
For any multi-indices $\alpha \in \mathbb{N}^3$ with $\left |\alpha \right| \leq s-1$, applying $\pa_x^\al$ to the momentum equation for ion in \eqref{main}, taking the inner product with $\pa_x^\al u_i$ in $L^2$, we have
\begin{eqnarray}\label{startui}
\dfrac{1}{2}\dfrac{d}{dt}\|\pa_x^\al u_i\|^2+\|\pa_x^\al u_i\|^2=-\left<\pa_x^\al\left((u_i\cdot \nabla)u_i\right),\pa_x^\al u_i\right>+\left<\eps^2\pa_x^\al (\D\vp-\nabla h_i(\rho_i)),\pa_x^\al u_i\right>.
\end{eqnarray}
By theorem \ref{Thm2.2}, Cauchy-Schwarz inequality and Moser-type inequality, we have
\begin{eqnarray*}
    &&\left|\left< \pa_x^\al\left( (u_i\cdot \nabla)u_i\right ) ,\pa_x^\al u_i\right>\right|\\ 
    &\leq &  C \left\|\nabla u_i^\eps \right\|^2_{s-2} \left\|u_i^\eps \right\|_{s-1}+
    \frac{1}{2} \left|\left\langle \dive(u_i^\eps), (\partial^\al_x u_i)^2 \right\rangle\right|\\
    &\leq & C \left\|u_i^\eps\right\|^3_{s-1}.
\end{eqnarray*}
By Young's inequality we have 
\begin{equation*}
   \left|\left<\eps^2\pa_x^\al (\D\vp-\nabla h_i(\rho_i)),\pa_x^\al u_i\right>\right|\leq \frac{1}{2} \left \| \pa_x^\al u_i \right \|^2+C\eps^4 \left \| \D\rho_i \right \|_{s-1}^2+C\eps^4\|\D\vp\|_s^2.
\end{equation*}
Substituting these estimates into \eqref{startui}, we conclude
\[
\dfrac{1}{2}\dfrac{d}{dt}\|\pa_x^\al u_i\|^2+\dfrac{1}{2}\|\pa_x^\al u_i\|^2\leq \left\|u_i\right\|^3_{s-1}+C\eps^4 \left \| \D\rho_i \right \|_{s-1}^2+C\eps^4\|\D\vp\|_s^2.
\]
Summing the above estimate for all $|\al|\leq s-1$, and integrating the resulting equation over $[0,t]$, we have
\begin{equation*}
    \left \| u_i(t) \right \|_{s-1}^{2}+\int_{0}^{t} \, \left \| u_i(\tau) \right \|_{s-1}^{2}d\tau \leq C\eps^4,
\end{equation*}
provided that $\|u_i\|_{s-1}$ is sufficiently small. This ends the proof. 
\end{proof}

From the above lemma and the equation \eqref{main}, we know that sequences $(\eps^{-2}u_i)_{\eps>0}$ and $(\pt(\eps^{-2}u_i))_{\eps>0}$ are uniformly bounded. Consequently, there exists a function $\bar{u}_i\in L^\infty(\R^+; H^{s})$ such that as $\eps\to 0$, 
\begin{eqnarray}\label{converui}
\eps^{-2}u_i&\rightharpoonup& \bar{u}_i, \quad \text{weakly-* in } \,\,\, L^\infty(\R^+;H^s),
\end{eqnarray}
and $(\eps^{-2}u_i^\eps)_{\eps>0}$ is relatively compact in $C([0,T];H_{loc}^{s_1})$, for all $s_1\in(0,s)$. These are sufficient for us to pass the limit in the momentum equation for ions to get its limiting equation as 
\begin{equation}\label{bu}
\pt \bar{u}_i+\bar{u}_i=\D\bar{\vp}, \quad \bar{u}_i(0,x)=\bar{u}_{i,0}(x),
\end{equation}
where $\bar{u}_{i,0}(x)$ is the weak limit of $\eps^{-2}u_{i,0}$ in $H^{s-1}$. 

Let $\bar{\rho}_i^1$ be the solution of the following equation
\begin{equation}\label{4.1}
    \pt \bar{\rho}_i^1+\dive \bar{u}_i=\Delta \bar{\rho}_i^1, \quad \bar{\rho}_i^1(0,x)=\bar{\rho}_{i,0}^1(x),
\end{equation}
and impose \begin{equation} \int_{\mathbb{R}^3} \bar{\rho}_{i,0}^1(x) dx=0.\end{equation}
We then have the following estimates for the first-order profile of the equation.
\begin{lemma}\label{regurhoi1} (Existence and global estimates on first profile systems) Let $s\geq 3$ be an integer. Then there exists positive constants $\delta$ and $C$ such that if
\[
\|\bar{\rho}_{i,0}^1\|_{s-1}+\|\bar{u}_{i,0}\|_s\leq C\delta,
\]
then there exists unique solution $(\bar{\rho}_i^1,\bar{u}_i)$ to \eqref{bu}-\eqref{4.1} satisfying the following energy estimates
\begin{equation}
\|\bar{u}_i(t)\|_s^2+\int_0^t \|\bar{u}_i(\tau)\|_s^2 d\tau \leq C\delta,
\end{equation}
and
\begin{equation}
\|\bar{\rho}_i^1(t)\|_{s-1}^2+\int_0^t \|\D \bar{\rho}_i^1(\tau)\|_{s-1}^2 d\tau \leq C,
\end{equation}
moreover for $|\al| \leq 2s-2$, it holds,
\begin{equation}
    \int_0^t \int_{\mathbb{R}^3} |\pa^\al_x \D \Delta \bar{\rho}_i^1(s,x)| dx ds \leq C , \  \ \forall t>0.
\end{equation}

\end{lemma}
\begin{proof} The existence of $\bar{u}_i$ is clear according to \eqref{converui}. According to the equation \eqref{bu} and Proposition \ref{danjiep}, it is not difficult to establish the following estimates for $\bar{u}_i$:
\[
\|\bar{u}_i(t)\|_s^2+\int_0^t \|\bar{u}_i(\tau)\|_s^2 d\tau \leq C\delta. 
\]
In addition, by Proposition \ref{danjiep}, we have
\[
\|\bar{\rho}_e-1\|_s^2\leq C\delta e^{-q t},
\]
where $q>0$ is a positive constant  and $\delta>0$ is sufficiently small. Let $|\al|\leq s-1$ be an multi-index. Taking the $\dive\pa_x^\al $ operator to both sides of the equation \eqref{bu} and combining the limiting equation, we have
\[
\pt \dive \pa_x^\al\bar{u}_i+\pa_x^\al\dive \bar{u}_i=\Delta \pa_x^\al\bar{\vp}=\pa_x^\al(1-\bar{\rho}_e). 
\]
Consequently, the exact solution to the above equation is of the form
\[
\dive \pa_x^\al\bar{u}_i(x,t)= e^{-t}\dive \pa_x^\al\bar{u}_{i,0}+e^{-t}\int_0^t(\pa_x^\al(1-\bar{\rho}_e(\tau)))e^{\tau}d\tau,
\]
which implies there exists a positive constant $q^\prime$, such that
\begin{eqnarray}
\|\dive \bar{u}_i\|_{s-1}&\leq&\sum_{|\al|\leq s-1}\|\dive \pa_x^\al\bar{u}_i(x,t)\|\nonumber\\
&\leq& Ce^{-t}\|\dive\bar{u}_{i,0}\|_s+Ce^{-t}\int_0^t\|(1-\bar{\rho}_e(\tau))\|_se^\tau d\tau \nonumber\\
&\leq&  C\delta e^{-t}+Ce^{-t}\int_0^t e^{-(q/2-1)\tau}d\tau\nonumber\\
&\leq& C\delta e^{-t}+\dfrac{C\delta }{q/2-1} (e^{-t}-e^{-qt/2})\nonumber\\
&\leq& C\delta e^{-q^\prime t},\nonumber
\end{eqnarray}
which shows the exponential decay property of $\|\dive \bar{u}_i\|_{s-1}$ when $t\to+\infty$. 

We now establish the estimates for $\bar{\rho}_{i}^1$. Since $\bar{u}_i$ is given, we have the exact solution to the periodic problem \eqref{4.1} is of the following form
\[
\bar{\rho}_i^1(t,x)=\int_{\R^3} \bar{\rho}_{i,0}^1(y)G(t,y;x)dy- \int_0^t\int_{\R^3}\dive\bar{u}_i(y,\tau)G(t-\tau,y;x)dyd\tau,
\]
where the fundamental solution
\[
G(t,y;x)=\dfrac{1}{(4\pi t)^{\frac{3}{2}}}\exp\left(-\dfrac{|x-y|^2}{4t}\right),\quad \exp(z)=e^z.
\]

Energy estimates for \eqref{4.1} show directly that
\[
\|\bar{\rho}_i^1(t)\|_s^2+\int_0^t\|\D\bar{\rho}_i^1(\tau)\|_s^2d\tau\leq C\int_0^t\|\bar{u}_i(\tau)\|_s^2d\tau\leq C\delta.
\]
The above estimate is not sufficient for our proof. We now establish the $L^1$ estimate for $\Delta \bar{\rho}_{i}^1$. Since for a certain $j=1,2,3$,
\[
\int_{\R^3}|\pa_{x_j}\partial^\alpha_x \Delta_x G(t,y;x)|dy\leq \dfrac{C}{t^{\frac{3}{2}}}
\]
and for all $|\alpha|\leq 2s-2$ when $t\geq 1$.
On the other hand, we have
\begin{eqnarray}
|\pa_{x_j}\pa_x^\al \Delta\bar{\rho}_i^1(t,x)|&\leq&\int_{\R^3} |\bar{\rho}_{i,0}^1(y)||\pa_{x_j} \partial^\alpha_x \Delta G(t,y;x)|dy+ \int_0^t\int_{\R^3}|\dive\bar{u}_i(y,\tau)||\pa_{x_j} \partial^\alpha_x \Delta G(t-\tau,y;x)|dyd\tau\nonumber\\
&\leq& \|\bar{\rho}_{i,0}^1\|_{s-1}\dfrac{C}{t^{\frac{3}{2}}}+\int_0^t\|\dive\bar{u}_i\|_{s-1}\dfrac{C}{(t-\tau)^{\frac{3}{2}}}d\tau\nonumber\\
&\leq&\|\bar{\rho}_{i,0}^1\|_{s-1}\dfrac{C}{t^{\frac{3}{2}}}+\int_0^t\dfrac{C e^{-q'\tau}}{(t-\tau)^{\frac{3}{2}}}d\tau.
\end{eqnarray}

Notice that 
\[
\int_0^t\dfrac{C e^{-q'\tau}}{(t-\tau)^{\frac{3}{2}}}d\tau=C e^{-q't}\int_0^t\dfrac{e^{q's}}{s^{\frac{3}{2}}}ds\to \dfrac{C }{q'}{t^{-\frac{3}{2}}}, \quad \text{as}\,\, t\to\infty,
\]

in this way we have the decay rate, combining with the existence of local solution we have

\[
\int_0^t \int_{\mathbb{R}^3}  |\pa_{x_j}\pa_x^\al \Delta\bar{\rho}_i^1(s,x)| dx ds \leq C, 
\]
which ends the proof.

\end{proof}

After all these preparations, we tend to establish the estimates for the following error variables
\[
\mathcal{N}_\nu=\rho_\nu-\bar{\rho}_\nu, \quad w_i=\eps^{-2}u_i-\bar{u}_i, \quad w_e=u_e-\bar{u}_e, \quad \mathcal{F}=\D\vp-\D\bar{\vp}. 
\]

\subsection{Error Estimates based on Anti-symmetric Structure}
As before, we introduce the following dissipative energy
\begin{eqnarray*}
\mathcal{D}(t)&=&\|\Ne(t)\|_{s-1}^2+\|\Ni(t)\|_{s-1}^2+\|w_e(t)\|_{s-1}^2+\|\calf(t)\|_{s-1}^2.
\end{eqnarray*}
We first have the estimates for error variables for ions.
\begin{lemma}(Estimates for ion variables) It holds
\begin{equation}\label{regui}
\|\Ni(t)\|_{s-1}^2+\|\eps w_i(t)\|_{s-1}^2+\int_0^t\|\eps w_i(\tau)\|_{s-1}^2d\tau\leq C\eps^2, \quad \forall\, t>0.
\end{equation}
\end{lemma}
\begin{proof} By the uniform estimate obtained in Theorem \ref{Thm2.2}, we have
\[
\|\Ni(t)\|_{s-1}^2\leq C\eps^2. 
\]
In addition, by the definition of $w_i$, it is clear that
\[
\|\eps w_i(t)\|_{s-1}\leq C\|\eps^{-1} u_i\|_{s-1}+C\|\eps \bar{u}_i\|_{s-1},
\]
which, combining \eqref{finalpre2} and lemma \ref{regurhoi1}, ends the proof. 
\end{proof}

\begin{lemma}\label{14} (Estimates for electron variables) It holds
\begin{equation}\label{regue}
    \|(\Ne,w_e,\mathcal{F})(t)\|_{s-1}^2 + \int_{0}^{t}\left\|w_e(\tau)\right\|^2_{s-1} d\tau 
    \leq  C\eps^2+C\delta\int_0^t\mathcal{D}(\tau)d\tau.
\end{equation}
\end{lemma}

\begin{proof} The error system for the electrons are of the form
\begin{equation*}
\begin{cases}
 \pt \mathcal{N}_e +\rho_e \dive w_e+\nabla\mathcal{N}_e\cdot u_e+\nabla\bar{\rho}_e\cdot w_e=- \mathcal{N}_e \dive\bar{u}_e,\\
 \pt w_e+(u_e \cdot \nabla) w_e+h'_e(\rho_e)\nabla\mathcal{N}_e+\mathcal{N}_e\nabla h'_e(\bar{\rho}_e)=-(w_e \cdot \nabla) \bar{u}_e
     -\mathcal{F}-w_e-r(\bar{\rho}_e,\mathcal{N}_e),
\end{cases}    
\end{equation*}
where
\begin{equation*}
    r(\bar{\rho}_e,\mathcal{N}_e)=(h'_e(\rho_e)-h'_e(\bar{\rho}_e)-h''_e(\bar{\rho}_e)\mathcal{N}_e)\nabla \bar{\rho}_e.
\end{equation*}

It can be written into the following first-order quasi-linear system:
\begin{equation}\label{16}
    \pt \mathcal{V}+\sum_{j=1}^3 A_j(\rho_e,u_e) \partial_{x_j} \mathcal{V}+ L(\bar{\rho}_e)\mathcal{V}=f,
\end{equation}
with 
\begin{equation*}
    \mathcal{V}=\begin{pmatrix}
    \mathcal{N}_e \\
    w_e
    \end{pmatrix}, \ \ \
    A_j(\rho_e,u_e)=\begin{pmatrix}
    u_{e}^j & \rho_e \xi_j^\top\\
    h'_e(\rho_e)\xi_j& u_{e}^j {\bf{I}}_3
    \end{pmatrix},\ \ \ 
    j=1,2,3,
\end{equation*}
and
\begin{equation*}
    L(\bar{\rho}_e)=\begin{pmatrix}
    0 & (\nabla \bar{\rho}_e)^\top\\
    \nabla h'_e(\bar{\rho}_e) & 0
    \end{pmatrix}, \quad 
    f=\begin{pmatrix}
    - \mathcal{N}_e \dive\bar{u}_e\\
    -(w_e \cdot \nabla) \bar{u}_e
     -\mathcal{F}-w_e-r(\bar{\rho}_e,\mathcal{N}_e)
     \end{pmatrix},
\end{equation*}
where $u_e=(u_e^1, u_e^2, u_e^3)$, ${\bf{I}}_3$ is the $3\times 3$ unit matrix and $\{\xi_k\}_{k=1}^3$ is the canonical basis of $\R^3$. 
For all $|\alpha| \leq s-1$, applying $\pa_x^\al$ to \eqref{16}, we obtain
\begin{equation}\label{17}
    \pt \pa_x^\al \mathcal{V}+\sum_{j=1}^3 A_j(\rho_e,u_e) \partial_{x_j} \pa_x^\al \mathcal{V}+ L(\bar{\rho}_e)\pa_x^\al \mathcal{V}=\pa_x^\al f+g^\alpha,
\end{equation}
in which
\begin{equation*}
    g^\alpha=\sum_{j=1}^3 A_j(\rho_e,u_e) \partial_{x_j} \pa_x^\al \mathcal{V}-\pa_x^\al \left ( \sum_{j=1}^3 A_j(\rho_e,u_e) \partial_{x_j} \mathcal{V} \right )+L(\bar{\rho}_e)\pa_x^\al \mathcal{V}-\pa_x^\al \left (L(\bar{\rho}_e)\mathcal{V} \right ).
\end{equation*}
Defining the symmetrizer $A_0(\rho_e)$ as follows
\begin{equation*}
    A_0(\rho_e)=\begin{pmatrix}
    h'_e(\rho_e) & 0\\
    0 & \rho_e {\bf{I}}_3
    \end{pmatrix},
\end{equation*}
and taking the inner product of \eqref{17} with $2A_0(\rho_e)\pa_x^\al \mathcal{V}$ in $L^2$ yields the following energy equality:
\begin{equation}\label{18}
\begin{aligned}
    \frac{d}{dt} \left< A_0(\rho_e)\pa_x^\al \mathcal{V}, \pa_x^\al \mathcal{V} \right> = & \left< \pt A_0(\rho_e)\pa_x^\al \mathcal{V}, \pa_x^\al \mathcal{V} \right>+\left<  B(\mathcal{V},\nabla\mathcal{V})\pa_x^\al \mathcal{V}, \pa_x^\al \mathcal{V} \right>\\
    & +2\left<  A_0(\rho_e)g^\alpha, \pa_x^\al \mathcal{V} \right>+2\left<  A_0(\rho_e)\pa_x^\al f, \pa_x^\al \mathcal{V} \right>\\
    \overset{def}{=}& J_1^\alpha+J_2^\alpha+J_3^\alpha+J_4^\alpha,
\end{aligned}
\end{equation}
with the natural correspondence of $J_1^\alpha$ to $J_4^\alpha$ and the matrix $B(\mathcal{V},\nabla\mathcal{V})$ is defined as follows

\begin{equation*}
\begin{aligned}
    B(\mathcal{V},\nabla\mathcal{V})&=\begin{pmatrix}
    \dive(h'_e(\rho_e)u_e) & (\nabla p'(\rho_e)-2h'_e(\rho_e)\nabla\bar{\rho}_e)^\top\\
    \nabla p'(\rho_e)-2\rho_e \nabla h'_e(\bar{\rho}_e) & \dive(\rho_e u_e){\bf{I}}_3
    \end{pmatrix}\\
    &\overset{def}{=} \begin{pmatrix}
    B_{11} & B_{12} \\
    B_{21}& B_{22}
    \end{pmatrix}.
\end{aligned}
\end{equation*}
For $J_1^\alpha$, we have
\begin{equation}\label{19}
| J_1^\alpha| \leq \left\|\pt A_0(\rho_e) \right\|_\infty \left\|\pa_x^\al \mathcal{V} \right\|^2 \leq C\delta \mathcal{D}(t).
\end{equation}
For $J_2^\alpha$, direct calculations give
\begin{equation*}
    J_2^\alpha=\left< B_{11} \pa_x^\al\mathcal{N}_e ,\pa_x^\al \mathcal{N}_e \right>+\left< B_{22} \pa_x^\al w_e ,\pa_x^\al w_e \right>+\left< (B_{12}+B_{21}^\top) \pa_x^\al\mathcal{N}_e ,\pa_x^\al w_e \right>,
\end{equation*}
where
\begin{equation}\label{21}
    \left< B_{11} \pa_x^\al\mathcal{N}_e ,\pa_x^\al \mathcal{N}_e \right> \leq 
    \left\| B_{11} \right\|_\infty \left\| \pa_x^\al \mathcal{N}_e \right\|^2 \leq C\delta \left\| \pa_x^\al \mathcal{N}_e \right\|^2,
\end{equation}
and
\begin{equation}\label{22}
    \left< B_{22} \pa_x^\al w_e ,\pa_x^\al w_e \right> \leq 
    \left\| B_{22} \right\|_\infty \left\| \pa_x^\al w_e \right\|^2 \leq C\delta \left\| \pa_x^\al w_e \right\|^2.
\end{equation}

From Taylor Expansion we have,
\begin{equation*}
    B_{12}(\rho_e,\nabla\rho_e)=B_{12}(\bar{\rho}_e,\nabla \bar{\rho}_e)+\frac{\partial B_{12}(\rho_e^a,\nabla \rho_e^b)}{\partial \rho_e}\mathcal{N}_e+\frac{\partial B_{12}(\rho_e^a,\nabla \rho_e^b)}{\partial (\nabla \rho_e)}\nabla \mathcal{N}_e,
\end{equation*}
and
\begin{equation*}
    B_{21}(\rho_e,\nabla\rho_e)=B_{21}(\bar{\rho}_e,\nabla \bar{\rho}_e)+\frac{\partial B_{21}(\rho_e^c,\nabla \rho_e^d)}{\partial \rho_e}\mathcal{N}_e+\frac{\partial B_{21}(\rho_e^c,\nabla \rho_e^d)}{\partial (\nabla \rho_e)}\nabla \mathcal{N}_e,
\end{equation*}
with $\rho_e^a$, $\rho_e^b$, $\rho_e^c$ and $\rho_e^d$ being between $\rho_e$ and $\bar{\rho}_e$.
Note that,
\begin{equation*}
\frac{\partial B_{12}}{\partial \rho_e}=\left( p_e'''(\rho_e)\nabla \rho_e-2h''_e(\rho_e) \nabla \bar{\rho}_e \right )^\top,\quad 
\frac{\partial B_{12}}{\partial (\nabla \rho_e)}=p_e''(\rho_e){\bf{I}}_3,
\end{equation*}
and
\begin{equation*}
\frac{\partial B_{21}}{\partial \rho_e}= p_e'''(\rho_e)\nabla \rho_e-2\nabla h'_e(\rho_e),\quad 
\frac{\partial B_{21}}{\partial (\nabla \rho_e)}=p_e''(\rho_e){\bf{I}}_3.
\end{equation*}
Also since
\begin{equation*}
    \left (B_{12}(\bar{\rho}_e,\nabla \bar{\rho}_e) \right)^\top=-B_{21}(\bar{\rho}_e,\nabla \bar{\rho}_e),
\end{equation*}
we have
\begin{equation}\label{20}
    \left< (B_{12}+B_{21}^\top) \pa_x^\al\mathcal{N}_e ,\pa_x^\al w_e \right> \leq \left\|B_{12}+B_{21}^\top \right\|_\infty \left\|\pa_x^\al \mathcal{N}_e\right\| \left\|\pa_x^\al w_e\right\| \leq C\delta \left (\left\|\pa_x^\al \mathcal{N}_e\right\|^2+\left\|\pa_x^\al w_e\right\|^2 \right).
\end{equation}
Combining \eqref{21}, \eqref{22} and \eqref{20} we conclude
\begin{equation}\label{23}
    |J_2^\alpha| \leq C\delta \mathcal{D}(t).
\end{equation}

For $J_3^\alpha$, by Moser-type inequality and Young's inequality we have
\begin{eqnarray}\label{24}
        &&\frac{1}{2}|J_3^\alpha| \nonumber \leq C\delta \mathcal{D}(t).
\end{eqnarray}

For $J_4^\alpha$, by Moser-type calculus inequalities,  we have
\begin{eqnarray}\label{31}
         J_4^\alpha \leq  -\left\langle 2\rho_e \pa_x^\al \calf , \pa_x^\al w_e \right\rangle-  \left\|\pa_x^\al w_e \right\|^2- C\delta\mathcal{D}(t). 
\end{eqnarray}
Noticing the equivalence of $\left< A_0(\rho_e)\pa_x^\al \mathcal{V}, \pa_x^\al \mathcal{V} \right>$ and $\left\|\pa_x^\al \mathcal{V}\right\|^2$ as $A_0$ is uniformly bounded, integrating \eqref{18} over $[0,t]$ and combining \eqref{19}, \eqref{23}, \eqref{24} and \eqref{31} yields
\begin{equation}\label{N_e}
    \left\|\pa_x^\al \mathcal{N}_e(t)\right\|^2+\left\|\pa_x^\al w_e(t)\right\|^2+ \int_{0}^{t}\left\|\pa_x^\al w_e(\tau)\right\|^2 d\tau+2\int_{0}^{t} \left\langle \rho_e \pa_x^\al \calf , \pa_x^\al w_e \right\rangle d\tau \leq  C\eps^2+C\delta \int_0^t \mathcal{D}(\tau)d\tau.
\end{equation}

Now it remains to estimate the last term on the left hand side of the above inequality. Since
\[
\left<\rho_e\pa_x^\al\calf,\pa_x^\al w_e\right>=\left<\pa_x^\al (\rho_e w_e), \pa_x^\al \calf\right>-\left<\pa_x^\al (\rho_e w_e)-\rho_e\pa_x^\al w_e,\pa_x^\al\calf\right>,
\]
in which 
\[
\left|\left<\pa_x^\al (\rho_e w_e)-\rho_e\pa_x^\al w_e,\pa_x^\al\calf\right>\right|\leq C\|\D \rho_e\|_{s-1}\|w_e\|_{s-1}\|\pa_x^\al \calf\|\leq C\delta \mathcal{D}(t).
\]
In addition, 
\begin{eqnarray*}
\left<\pa_x^\al (\rho_e w_e), \pa_x^\al \calf\right>=\left<\pt\pa_x^\al\Ne,\pa_x^\al(\vp-\bar{\vp})\right>-\left<\pa_x^\al(\Ne\bar{u}_e),\pa_x^\al \calf\right>,
\end{eqnarray*}
in which 
\[
\left|\left<\pa_x^\al(\Ne\bar{u}_e),\pa_x^\al \calf\right>\right|\leq C\|\bar{u}_e\|_s\|\Ne\|_{s-1}\|\pa_x^\al \calf\|\leq C\delta\mathcal{D}(t),
\]
and
\begin{eqnarray*}
\left<\pt\pa_x^\al\Ne,\pa_x^\al(\vp-\bar{\vp})\right>&=&\left<\pt\pa_x^\al(\Ne-\Ni),\pa_x^\al(\vp-\bar{\vp})\right>+\left<\pa_x^\al(\rho_iu_i),\pa_x^\al\calf\right>\nonumber\\
&\leq& -\dfrac{1}{2}\dfrac{d}{dt}\|\pa_x^\al\calf\|^2+C\|\eps^{-1}u_i\|_{s-1}^2+C\eps^2\|\calf\|_{s-1}^2.
\end{eqnarray*}
Substituting these estimates into \eqref{N_e} and summing the resulting equation for all $|\al|\leq s-1$ yield \eqref{regue}. 
\end{proof}

\subsection{Application of the Stream Function Technique} It remains to establish the dissipative estimates for $\Ni$ and $\Ne$, in which the stream function technique is necessary. We first construct the stream functions. 

We write equation \eqref{4.1} into the following
\[
\pt \bar{\rho}_i^1+\dive (\bar{u}_i - \nabla \bar{\rho}_i^1)=0.
\]
Subtracting the above equation multiplied by $\eps^2$ from the mass conservative equation for ions leads
\begin{equation*}
    \partial_t(\Ni-\eps^2\bar{\rho}_i^1)+\dive(\rho_i u_i-\eps^2\bar{u}_i + \eps^2 \nabla \bar{\rho}_i^1)=0.
\end{equation*}
in which further noticing that $u_i=\eps^2(w_i+\bar{u}_i)$, we have
\begin{equation*}
    \partial_t(\Ni-\eps^2\bar{\rho}_i^1)+\dive(\eps^2 \rho_i w_i + \eps^2\Ni\bar{u}_i + \eps^2 \nabla \bar{\rho}_i^1)=0.
\end{equation*}

Following the methods in constructing stream functions in \cite{Zhao2021a}, there exist a unique stream function $\psi_i$ and a remaining function $M_i$ satisfying
\begin{equation*}
\begin{cases}
\partial_t \psi_i=\eps^2 \rho_i w_i + \eps^2\Ni\bar{u}_i + \eps^2 \nabla \bar{\rho}_i^1 +\nabla \times M_i,\\
\dive\psi_i=-\Ni + \eps^2 \bar{\rho}_i^1, \quad \nabla \times \psi_i=0,\\
\displaystyle\int_{\T^3} \psi_i dx=0.
\end{cases}
\end{equation*}
By Poincar\'e inequality, we have
\[
\|\psi_i\|_{s-1}^2\leq C\|\D\psi_i\|_{s-1}^2\leq C\|\dive\psi_i\|_{s-1}^2+C\|\D\times \psi_i\|_{s-1}^2\leq C\|\Ni\|_{s-1}^2+C\eps^4\|\bar{\rho}_i^1\|_{s-1}^2.
\]
Similarly, the error conservative equation for electrons is of the form
\[
\pt \Ne+\dive(\rho_eu_e-\bar{\rho}_e\bar{u}_e)=0,
\]
which implies that there exist a unique stream function $\psi_e$ and a remaining function $M_e$ satisfying
\[
\begin{cases}
\pt\psi_e=\rho_e u_e-\bar{\rho}_e\bar{u}_e+\nabla \times M_e,\\
\dive\psi_e=-\mathcal{N}_e,\quad \nabla \times \psi_e=0,\\
\displaystyle\int_{\mathbb{T}^3} \psi_e dx=0.
\end{cases}
\]

\subsection*{Applications to the Ion System} Now we apply the stream function technique to the error system. We first treat the ion system. The error equation for the momentum equation for ions reads
\[
\partial_t w_i+w_i+\eps^{-2}(u_i \cdot \nabla) u_i+\nabla h_i(\rho_i)=\mathcal{F}.
\]
For multi-indices $\al\in\N^3$ with $|\al|\leq s-1$, applying $\partial^\al_x$ to both sides of and taking the inner product of the resulting equation with $\rho_i \partial^\al_x \psi_i$, then we have
\begin{eqnarray*}
    &&\left\langle \pa_x^\al \calf, \rho_i \partial^\al_x \psi_i\right\rangle\nonumber\\
    & =&\left\langle \partial_t \partial^\al_x w_i, \rho_i \partial^\al_x \psi_i \right\rangle + \left\langle \partial^\al_x w_i, \rho_i \partial^\al_x \psi_i \right\rangle+ \eps^{-2} \left\langle \partial^\al_x ((u_i \cdot \nabla) u_i), \rho_i \partial^\al_x \psi_i \right\rangle + \left\langle \partial^\al_x \nabla h_i(\rho_i), \rho_i \partial^\al_x \psi_i \right\rangle\nonumber\\
    & \overset{def}{=}&K_i^1+K_i^2+K_i^3+K_i^4,
\end{eqnarray*}
with the natural correspondence of $K_i^1$ to $K_i^4$, which are treated term by term in a series of lemmas as follows.

\begin{lemma}\label{lemma4.2}(Estimate of $K_i^1$)
For all $|\alpha| \leq s-1$, it holds
\begin{equation}\label{4.35}
     K_i^1 \geq \frac{d}{dt} \left\langle \rho_i \partial^\al_x w_i, \partial^\al_x \psi_i \right\rangle+\frac{1}{4\eps^2} \|\nabla \times \partial^\al_x M_i\|^2-C\eps^2(\mathcal{D}(t)+\|w_i\|_{s-1}^2+\|\nabla \bar{\rho}_i^1\|_{s-1}^2)-\frac{C}{\eps^2}\|\psi_i\|^2_{s-1}\|u_i\|^2_{s}.
\end{equation}
\end{lemma}
\begin{proof}
By integration by parts, we have
\[
K_i^1 = \left\langle \partial_t(\rho_i \partial^\al_x w_i), \partial^\al_x \psi_i \right\rangle - \left\langle \partial_t \rho_i \partial^\al_x w_i, \partial^\al_x \psi_i \right\rangle,
\]
in which 
\[
\left\langle \partial_t(\rho_i \partial^\al_x w_i), \partial^\al_x \psi_i \right\rangle= \frac{d}{dt} \left\langle \rho_i \partial^\al_x w_i, \partial^\al_x \psi_i \right\rangle - \left\langle \rho_i \partial^\al_x w_i, \partial_t \partial^\al_x \psi_i \right\rangle.
\]
Direct calculations give
\begin{equation*}
    \left|\left\langle \partial_t \rho_i \partial^\al_x w_i, \partial^\al_x \psi_i \right\rangle\right| \leq C\|u_i\|_{s}\|\pa_x^\al w_i\|\|\pa_x^\al \psi_i\|\leq C(\|\eps w_i\|_{s-1}^2+\frac{1}{\eps^2} \|\psi_i\|^2_{s-1}\|u_i\|^2_{s}).
\end{equation*}
Note that
\begin{eqnarray*}
   -\left\langle \rho_i \partial^\al_x w_i, \partial_t \partial^\al_x \psi_i \right\rangle &=& - \left\langle \rho_i \partial^\al_x w_i, \eps^2 \partial^\al_x(\rho_i w_i +\bar{u}_i \Ni) \right\rangle - \left\langle \rho_i \partial^\al_x w_i, \eps^2\partial^\al_x \nabla \bar{\rho}_i^1 \right\rangle - \left\langle \rho_i \partial^\al_x w_i, \nabla \times \partial^\al_x M_i \right\rangle\nonumber\\
   &\geq & -C\eps^2(\|w_i\|_{s-1}^2+\|\nabla \bar{\rho}_i^1\|_{s-1}^2 + \mathcal{D}(t))- \left\langle \rho_i \partial^\al_x w_i, \nabla \times \partial^\al_x M_i \right\rangle,
\end{eqnarray*}
where
\begin{equation*}
\begin{aligned}
-\left\langle \rho_i \partial^\al_x w_i, \nabla \times \partial^\al_x M_i \right\rangle = - \left\langle \partial^\al_x (\rho_i w_i), \nabla \times \partial^\al_x M_i \right\rangle+\left\langle \partial^\al_x (\rho_i w_i)-\rho_i \partial^\al_x w_i, \nabla \times \partial^\al_x M_i \right\rangle.
\end{aligned}
\end{equation*}
Since $\psi_i$ is rotation free,  by Young's inequality we have
\begin{eqnarray*}
    - \left\langle \partial^\al_x (\rho_i w_i), \nabla \times \partial^\al_x M_i \right\rangle &= & -\frac{1}{\eps^2} \left\langle \partial_t \partial^\al_x \psi_i, \nabla \times \partial^\al_x M_i \right\rangle + \left\langle \partial^\al_x (\bar{u}_i \Ni), \nabla \times \partial^\al_x M_i \right\rangle\nonumber\\
    && + \left\langle \partial^\al_x \nabla \bar{\rho}_i^1, \nabla \times \partial^\al_x M_i \right\rangle + \frac{1}{\eps^2} \left\langle \nabla \times \partial^\al_x M_i, \nabla \times \partial^\al_x M_i \right\rangle\nonumber\\
    &\geq & \frac{1}{2\eps^2} \|\nabla \times \partial^\al_x M_i\|^2-C\eps^2\mathcal{D}(t),\nonumber
\end{eqnarray*}
and by Moser-type inequalities and the Young's inequality, we have
\begin{equation*}
    \left|\left\langle \partial^\al_x (\rho_i w_i)-\rho_i \partial^\al_x w_i, \nabla \times \partial^\al_x M_i \right\rangle\right| \leq  C \|\Ni\|_{s-1} \|w_i\|_{s-1} \|\nabla \times \partial^\al_x M_i\|  \leq  C\eps^2\mathcal{D}(t)+ \dfrac{1}{4\eps^2}\|\nabla \times \partial^\al_x M_i\|^2.    
\end{equation*}
Thus we have
\begin{equation*}
    -\left\langle \rho_i \partial^\al_x w_i, \nabla \times \partial^\al_x M_i \right\rangle \geq \frac{1}{4\eps^2} \|\nabla \times \partial^\al_x M_i\|^2-C\eps^2\mathcal{D}(t).
\end{equation*}
Consequently
\begin{equation*}
-\left\langle \rho_i \partial^\al_x w_i, \partial_t \partial^\al_x \psi_i \right\rangle \geq \frac{1}{4\eps^2} \|\nabla \times \partial^\al_x M_i\|^2-C\eps^2(\mathcal{D}(t)+\|\nabla \bar{\rho}_i^1\|_{s-1}^2+\|w_i\|^2_{s-1}).
\end{equation*}
Combining all these estimates implies \eqref{4.35}. 
\end{proof}

\begin{lemma}\label{lemma4.3}(Estimates for $K_i^2$)
For all $|\alpha| \leq s-1$, it holds
\begin{equation}\label{k2integral}
K_i^2\geq \frac{1}{2\eps^2} \frac{d}{dt} \|\partial^\al_x \psi_i\|^2-C\eps^2\mathcal{D}(t) - C\|\psi_i\|_{s-1}^2 \left(\dfrac{1}{\eps^2}(\|w_i\|_{s-1}^2+\|\bar{u}_i\|^2_{s-1}) + \int_{\mathbb{R}^3} |\pa_x^{2\al} \D \Delta \bar{\rho}^1_i (t,x)| dx \right).
\end{equation}
\end{lemma}

\begin{proof}
First we have
\begin{equation*}
    K_i^2=\left\langle \rho_i \partial^\al_x w_i, \partial^\al_x \psi_i\right\rangle = \left\langle \partial^\al_x (\rho_i w_i), \partial^\al_x \psi_i\right\rangle +\left\langle \rho_i \partial^\al_x w_i - \partial^\al_x (\rho_i w_i), \partial^\al_x \psi_i\right\rangle,
\end{equation*}
in which by the Moser-type calculus inequalities, 
\[
|\left\langle \rho_i \partial^\al_x w_i - \partial^\al_x (\rho_i w_i), \partial^\al_x \psi_i\right\rangle|\leq C\|\Ni\|_{s-1}\|w_i\|_{s-1}\|\pa_x^\al \psi_i\|\leq C\eps^2 \mathcal{D}(t)+\frac{C}{\eps^2}\|w_i\|^2_{s-1}\|\psi_i\|_{s-1}^2.
\]
Note that
\begin{eqnarray*}
        \left\langle \partial^\al_x (\rho_i w_i), \partial^\al_x \psi_i\right\rangle &= & \frac{1}{\eps^2} \left\langle \partial_t \partial^\al_x \psi_i, \partial^\al_x \psi_i \right\rangle + \left\langle \partial^\al_x (\bar{u}_i \Ni), \partial^\al_x \psi_i \right\rangle\nonumber\\
        && - \left\langle \partial^\al_x \nabla \bar{\rho}_i^1, \partial^\al_x \psi_i \right\rangle - \frac{1}{\eps^2} \left\langle \nabla \times \partial^\al_x M_i, \partial^\al_x \psi_i \right\rangle\nonumber\\
        &\geq& \frac{1}{2\eps^2} \frac{d}{dt} \|\partial^\al_x \psi_i\|^2-\frac{C}{\eps^2}\|\bar{u}_i\|^2_{s-1}\|\psi_i\|^2_{s-1}-C\eps^2\mathcal{D}(t) -  \left\langle \partial^\al_x \nabla \bar{\rho}_i^1, \partial^\al_x \psi_i \right\rangle,
\end{eqnarray*}
while by Poincar\'e inequality and integration by parts we have
\begin{eqnarray*}
     |\left\langle  \partial^\al_x \nabla \bar{\rho}_i^1, \partial^\al_x \psi_i \right\rangle| \leq \int_{\mathbb{R}^3} |\pa_x^\al \nabla \bar{\rho}_i^1 \cdot \pa_x^\al \psi_i | dx \leq C\|\psi_i\|_{s-1} \int_{\mathbb{R}^3} |\pa_x^{2\al} \D \Delta \bar{\rho}^1_i (t,x)| dx 
\end{eqnarray*}
Combining all these estimates yields \eqref{k2integral}.
\end{proof}

\begin{lemma}\label{lemma4.4} (Estimates for $K_i^3$ and $K_i^4$)
For all $|\alpha| \leq s-1$, it holds
\begin{equation}\label{k34integral}
K_i^3+K_i^4 \geq  h_1 \|\partial^\al_x \Ni\|^2 - (C\delta+C\eps^2)\mathcal{D}(t)-\dfrac{C}{\eps^2} \| u_i\|^2_{s-1}-C\eps^{2}\|\D \bar{\rho}_i^1\|_{s-1}^2-\dfrac{C}{\eps^{2}} \|u_i\|^2_{s} \|\psi_i\|^2_{s-1}.
\end{equation}
\end{lemma}
\begin{proof}
For $K_i^3$, by Young's inequality we have
\begin{equation*}
    |K_i^3|\leq  \eps^{-2} \|u_i\|_{s-1} \|u_i\|_{s} \|\psi_i\|_{s-1}
     \leq  \dfrac{C}{\eps^{2}}\|u_i\|_{s-1}^2 + \dfrac{C}{\eps^{2}} \|u_i\|^2_{s} \|\psi_i\|^2_{s-1}  
\end{equation*}
For $K_i^4$, we first notice that by Taylor expansion we have a $\rho_i^h$ between $\rho_i$ and $1$ satisfying
\[
h_i(\rho_i)-h_i(1)=h'_i(\rho_i^h) \Ni.
\]
Hence, 
\begin{equation*}
K_i^4 = -\left\langle \partial^\al_x (h'_i(\rho_i^h) \Ni),\dive(\rho_i \partial^\al_x \psi_i) \right\rangle= - \left\langle \partial^\al_x (h'_i(\rho_i^h) \Ni),\dive(\partial^\al_x \psi_i) \right\rangle - \left\langle \partial^\al_x (h'_i(\rho_i^h) \Ni),\dive(\Ni \partial^\al_x \psi_i) \right\rangle.
\end{equation*}
By the definition of the stream function, we have
\[
   - \left\langle \partial^\al_x (h'_i(\rho_i^h) \Ni),\dive(\partial^\al_x \psi_i) \right\rangle = \left\langle \partial^\al_x (h'_i(\rho_i^h) \Ni), \partial^\al_x\Ni \right\rangle - \eps^2  \left\langle \partial^\al_x (h'_i(\rho_i^h) \Ni), \partial^\al_x \bar{\rho}_i^1 \right\rangle,
\]
where
\begin{eqnarray*}
\left\langle \partial^\al_x (h'_i(\rho_i^h) \Ni), \partial^\al_x\Ni \right\rangle&= & \left\langle h'_i(\rho_i^h) \partial^\al_x \Ni, \partial^\al_x\Ni \right\rangle+\left\langle \partial^\al_x (h'_i(\rho_i^h) \Ni)-h'_i(\rho_i^h) \partial^\al_x \Ni, \partial^\al_x\Ni \right\rangle\\
 &\geq & h_1 \|\partial^\al_x \Ni\|^2 - C\delta \mathcal{D}(t),
\end{eqnarray*}
and by Poincar\'e inequality
\[
    \left|\left\langle \partial^\al_x (h'_i(\rho_i^h) \Ni), \eps^2 \partial^\al_x \bar{\rho}_i^1 \right\rangle\right| \leq C\eps^2\mathcal{D}(t)+ C \eps^2 \|\D \bar{\rho}_i^1\|^2_{s-1}.
\]
Also we have
\[
  \left|\left\langle \partial^\al_x(h'_i(\rho_i^h) \Ni) ,\dive(\Ni \partial^\al_x \psi_i) \right\rangle\right|\leq C\delta\mathcal{D}(t)+C\eps^2\mathcal{D}(t)+C\eps^2\|\D \bar{\rho}_i^1\|_{s-1}^2.
\]
Combining all these estimates, we have \eqref{k34integral}.
\end{proof}

\begin{lemma}(Dissipative estimates for $\Ni$) It holds
\begin{eqnarray}\label{Nifinal}
&&-\int_0^t\sum_{|\al|\leq s-1}\left\langle \pa_x^\al \calf, \rho_i \partial^\al_x \psi_i\right\rangle d\tau+\dfrac{1}{\eps^2}\|\psi_i(t)\|_{s-1}^2+\int_0^t\left(\dfrac{1}{\eps^2}\|\D\times M_i(\tau)\|_{s-1}^2+\|\Ni(\tau)\|_{s-1}^2\right)d\tau\nonumber\\
&\leq& C\eps^2+C\delta\int_0^t\mathcal{D}(\tau)d\tau+\dfrac{C\delta+c_0\mu_0}{\eps^2}\sup_{0\leq \tau \leq t}\|\psi_i(\tau)\|_{s-1}^2,
\end{eqnarray}
where $\mu_0>0$ is a sufficiently small constant, of which the value is determined in \eqref{mu0}.
\end{lemma}
\begin{proof}
Combining Lemmas \ref{lemma4.2} to \ref{lemma4.4}, we have
\begin{eqnarray}\label{ifinalbefore}
&&\frac{d}{dt} \left\langle \rho_i \partial^\al_x w_i, \partial^\al_x \psi_i \right\rangle+\frac{1}{2\eps^2} \frac{d}{dt} \|\partial^\al_x \psi_i\|^2+\frac{1}{4\eps^2} \|\nabla \times \partial^\al_x M_i\|^2+h_1 \|\partial^\al_x \Ni\|^2 - \left\langle \pa_x^\al \calf, \rho_i \partial^\al_x \psi_i\right\rangle\nonumber\\
&\leq& C\delta\mathcal{D}(t)+C\eps^2\mathcal{D}(t)+C\eps^2(\|w_i\|_{s-1}^2+\|\eps^{-2}u_i\|_{s-1}^2+\|\D\bar{\rho}_i^1\|_{s-1}^2))+ \dfrac{C}{\eps^2} \|u_i\|_{s-1}^2 \nonumber\\
&& + \dfrac{C}{\eps^2}\|\psi_i\|_{s-1}^2 \left(\|u_i\|_s^2 +\|w_i\|_{s-1}^2+\|\bar{u}_i\|^2_{s-1} \right) + C \|\psi_i\|_{s-1}^2 \int_{\mathbb{R}^3} |\pa_x^{2\al} \D \Delta \bar{\rho}^1_i (t,x)| dx .
\end{eqnarray}

By Poincar\'e inequality and Young's inequality we have
\begin{eqnarray*}\label{inftyat0}
    |\left\langle \rho_i \partial^\al_x w_i, \partial^\al_x \psi_i \right\rangle(0)| 
    \leq  C\eps^2,
\end{eqnarray*}
and
\begin{eqnarray*}\label{inftyatt}
    |\left\langle \rho_i \partial^\al_x w_i, \partial^\al_x \psi_i \right\rangle(t)| \leq  C\eps^2 + \frac{1}{2\eps^2} \|\partial^\al_x \psi_i(t)\|^2.
\end{eqnarray*}
Consequently, integrating \eqref{ifinalbefore} over $[0,t]$, summing the resulting equation for all $|\al|\leq s-1$ and applying Young's inequality yields
\begin{eqnarray*}
&&-\int_0^t\sum_{|\al|\leq s-1}\left\langle \pa_x^\al \calf, \rho_i \partial^\al_x \psi_i\right\rangle d\tau+\dfrac{1}{\eps^2}\|\psi_i(t)\|_{s-1}^2+\int_0^t\left(\dfrac{1}{\eps^2}\|\D\times M_i(\tau)\|_{s-1}^2+\|\Ni(\tau)\|_{s-1}^2\right)d\tau\nonumber\\
&\leq& C\eps^2+C\delta\int_0^t\mathcal{D}(\tau)d\tau + \dfrac{C\delta+c_0\mu_0}{\eps^2}\sup_{0\leq \tau \leq t}\|\psi_i(\tau)\|_{s-1}^2 + C\eps^2 \int_0^t \int_{\mathbb{R}^3} |\pa_x^{2\al} \D \Delta \bar{\rho}^1_i (\tau,x)| dx d\tau,
\end{eqnarray*}
where $\mu_0>0$ is a sufficiently small constant to be determined later.
Recall Lemma \ref{regurhoi1} we obtain \eqref{Nifinal}.
\end{proof}

\subsection*{Applications to the Electron System} The error system for the momentum equation for electrons reads
\[
\pt(\rho_e u_e-\bar{\rho}_e\bar{u}_e)+\dive(\rho_e u_e\otimes u_e-\bar{\rho}_e\bar{u}_e\otimes\bar{u}_e)+\D(p_e(\rho_e)-p_e(\bar{\rho}_e))=-(\rho_e\calf+\Ne\D\bar{\vp})-(\rho_e u_e-\bar{\rho}_e\bar{u}_e).
\]
Applying $\partial^\al_x$ to both sides of the above equation and taking the inner product of the resulting equation with $\partial^\al_x \psi_e$, we have
\begin{eqnarray*}
    -\left\langle \rho_e\pa_x^\al \calf,\partial^\al_x \psi_e\right\rangle&=& \left\langle \pt \partial^\al_x (\rho_e u_e-\bar{\rho}_e\bar{u}_e), \partial^\al_x \psi_e\right\rangle + \left\langle \partial^\al_x (\rho_e u_e-\bar{\rho}_e\bar{u}_e), \partial^\al_x \psi_e\right\rangle\nonumber\\
    &&+\left\langle \partial^\al_x \nabla (p_e(\rho_e)-p_e(\bar{\rho}_e)), \partial^\al_x \psi_e\right\rangle+\left<\pa_x^\al(\rho_e\calf)-\rho_e\pa_x^\al\calf, \pa_x^\al\psi_e\right>\nonumber\\
    &&+ \left\langle \partial^\al_x (\dive(\rho_e u_e\otimes u_e-\bar{\rho}_e\bar{u}_e\otimes\bar{u}_e), \partial^\al_x \psi_e\right\rangle +\left<\pa_x^\al(\Ne\D\bar{\vp}),\pa_x^\al\psi_e\right>\nonumber\\
    &\overset{def}{=} & K_e^1+K_e^2+K_e^3+K_e^4+K_e^5+K_e^6,
\end{eqnarray*}
with the natural correspondence of $K_e^1$ to $K_e^6$, of which the estimates are established in a series of lemmas below. 

\begin{lemma}\label{lem12}(Reformulations for $K_e^1$ and $K_e^2$). For all $|\al|\leq s-1$, it holds,
\begin{equation}\label{ke12}
    K_e^1+K_e^2=\dfrac{1}{2}\dfrac{d}{dt}\|\pa_x^\al \psi_e\|^2+\dfrac{d}{dt}\left<\partial^\al_x (\rho_e u_e-\bar{\rho}_e\bar{u}_e), \pa_x^\al\psi_e\right>+\|\D\times \pa_x^\al M_e\|^2-\|\pa_x^\al(\rho_e u_e-\bar{\rho}_e\bar{u}_e)\|^2.
\end{equation}
\end{lemma}
\begin{proof} By integration by parts, we have
\[
K_e^1=\dfrac{d}{dt}\left<\partial^\al_x (\rho_e u_e-\bar{\rho}_e\bar{u}_e), \pa_x^\al\psi_e\right>-\left<\pa_x^\al(\rho_e u_e-\bar{\rho}_e\bar{u}_e), \pt \pa_x^\al\psi_e\right>,
\]
in which
\[
\left<\pa_x^\al(\rho_e u_e-\bar{\rho}_e\bar{u}_e), \pt \pa_x^\al\psi_e\right>=\|\pa_x^\al(\rho_e u_e-\bar{\rho}_e\bar{u}_e)\|^2+\left<\pa_x^\al(\rho_e u_e-\bar{\rho}_e\bar{u}_e), \D\times\pa_x^\al M_e\right>.
\]
Notice that
\[
\left<\pa_x^\al(\rho_e u_e-\bar{\rho}_e\bar{u}_e), \D\times\pa_x^\al M_e\right>=\left<\pt\pa_x^\al \psi_e, \D\times\pa_x^\al M_e\right>-\|\D\times \pa_x^\al M_e\|^2=-\|\D\times \pa_x^\al M_e\|^2,
\]
combining these identities, we obtain 
\begin{equation*}
    K_e^1=\dfrac{d}{dt}\left<\partial^\al_x (\rho_e u_e-\bar{\rho}_e\bar{u}_e), \pa_x^\al\psi_e\right>+\|\D\times \pa_x^\al M_e\|^2-\|\pa_x^\al(\rho_e u_e-\bar{\rho}_e\bar{u}_e)\|^2.
\end{equation*}

Noticing the definition of the stream function technique, we have
\[
K_e^2=\left<\pt\pa_x^\al \psi_e,\pa_x^\al\psi_e\right>-\left<\D\times\pa_x^\al M_e, \pa_x^\al \psi_e\right>=\dfrac{1}{2}\dfrac{d}{dt}\|\pa_x^\al \psi_e\|^2,
\]
which, combining the reformulation for $K_e^1$, yields \eqref{ke12}.
\end{proof}

\begin{lemma}\label{lem3}(Estimates for $K_e^3$). For all $|\al|\leq s-1$, it holds,
\begin{equation}\label{ke3}
    K_e^3\geq \dfrac{2}{3}h_1\|\pa_x^\al \Ne\|^2-C\delta\mathcal{D}(t).
\end{equation}
\end{lemma}
\begin{proof} We first notice that by Taylor expansion we have a $\rho_e^p$ between $\rho_e$ and $\bar{\rho_e}$ satisfying
\[
p_e(\rho_e)-p_e(\bar{\rho}_e)=p^\prime_e(\rho_e^p) \Ne.
\]
Hence, 
\begin{eqnarray}
K_e^3 &=& -\left\langle \partial^\al_x (p^\prime_e(\rho_e^p) \Ne),\dive(\partial^\al_x \psi_e) \right\rangle\nonumber\\
&=& \left<p^\prime_e(\rho_e^p)\pa_x^\al \Ne, \pa_x^\al N_e\right>+\left<\partial^\al_x (p^\prime_e(\rho_e^p) \Ne)-p^\prime_e(\rho_e^p)\pa_x^\al \Ne, \pa_x^\al N_e\right>,\nonumber
\end{eqnarray}
in which
\[
\left<p^\prime_e(\rho_e^p)\pa_x^\al \Ne, \pa_x^\al N_e\right>\geq \dfrac{2}{3}h_1\|\pa_x^\al \Ne\|^2,
\]
and
\[
\left|\left<\partial^\al_x (p^\prime_e(\rho_e^p) \Ne)-p^\prime_e(\rho_e^p)\pa_x^\al \Ne, \pa_x^\al N_e\right>\right|\leq C\delta\mathcal{D}(t).
\]
These estimates imply \eqref{ke3}. 
\end{proof}

\begin{lemma}\label{lem4567}(Estimates for $K_e^4$, $K_e^5$ and $K_e^6$) It holds
\begin{equation}\label{ke4567}
    |K_e^4+K_e^5+K_e^6|\leq C\delta\mathcal{D}(t).
\end{equation}
\end{lemma}
\begin{proof} Noticing that by the Poincar\'e inequality, 
\[
\|\psi_e\|_{s-1}^2\leq C\|\D \psi_e\|_{s-1}^2\leq C\|\dive \psi_e\|^2+C\|\D\times \psi_e\|_{s-1}^2\leq C\|\Ne\|_{s-1}^2,
\]
by Moser-type calculus inequalities, direct calculations give
\begin{eqnarray*}
|K_e^4|&\leq& C\|\D \rho_e\|_{s-1}\|\calf\|_{s-1}\|\Ne\|_{s-1},\\
|K_e^5|&\leq& C\|w_e\|_{s-1}^2\|\Ne\|_{s-1},\\
|K_e^6|&\leq& C\|\Ne\|_{s-1}\|\D\bar{\vp}\|_{s-1}\|\Ne\|_{s-1}.
\end{eqnarray*}
All these estimates imply \eqref{ke4567}.
\end{proof}

Combining Lemmas \ref{lem12}-\ref{lem4567}, we have
\begin{eqnarray*}
&&\dfrac{1}{2}\dfrac{d}{dt}\|\pa_x^\al \psi_e\|^2+\dfrac{d}{dt}\left<\partial^\al_x (\rho_e u_e-\bar{\rho}_e\bar{u}_e), \pa_x^\al\psi_e\right>+\|\D\times \pa_x^\al M_e\|^2+\dfrac{2}{3}h_1\|\pa_x^\al \Ne\|^2\nonumber\\
&\leq& C\|w_e\|_{s-1}^2+C\delta\mathcal{D}(t)-\left\langle \rho_e\pa_x^\al \calf,\partial^\al_x \psi_e\right\rangle.
\end{eqnarray*}
Integrating the above equation over $[0,t]$, summing the resulting equation for all $|\al|\leq s-1$ and combining the energy estimate for $w_e$ yield the following
\begin{eqnarray}\label{Nefinal}
    &&\int_0^t\sum_{|\al|\leq s-1}\left\langle \pa_x^\al \calf, \rho_e \partial^\al_x \psi_e\right\rangle d\tau+\|\psi_e(t)\|_{s-1}^2+\int_0^t\left(\|\D\times M_e\|_{s-1}^2+\|\Ne(\tau)\|_{s-1}^2d\tau\right)\nonumber\\
    &\leq& C\delta\int_0^t \mathcal{D}(\tau)d\tau+C\eps^2+\|w_e(t)\|_{s-1}^2+\int_0^t\|w_e(\tau)\|_{s-1
    }^2d\tau.
\end{eqnarray}

\begin{lemma}\label{Fpsi}It holds
\begin{equation}\label{Fintegral}
    \sum_{|\al|\leq s-1}\int_{0}^{t}  \left\langle \pa_x^\al \calf, \rho_i \partial^\al_x \psi_i\right\rangle - \left\langle \rho_e\pa_x^\al \calf, \partial^\al_x \psi_e\right\rangle d\tau \leq  C \eps^2 - \int_{0}^{t} \| \calf(\tau)\|_{s-1}^2 d\tau + C\delta \int_{0}^{t} \mathcal{D}(\tau)d\tau.  
\end{equation}
\end{lemma}
\begin{proof}
First we have
\begin{equation*}
    \left\langle \pa_x^\al \calf, \rho_i \partial^\al_x \psi_i\right\rangle - \left\langle \rho_e\pa_x^\al \calf, \partial^\al_x \psi_e\right\rangle= \left<\pa_x^\al \calf, (\rho_i-\rho_e)\pa_x^\al \psi_e\right>-\left<\pa_x^\al(\vp-\bar{\vp}),\dive(\rho_i\pa_x^\al(\psi_i-\psi_e))\right>,
\end{equation*}
in which 
\[
\left|\left<\pa_x^\al \calf, (\rho_i-\rho_e)\pa_x^\al \psi_e\right>\right|\leq C\|\pa_x^\al \calf\|\|\rho_i-\rho_e\|_{\infty}\|\pa_x^\al \Ne\|\leq C\delta\mathcal{D}(t).
\]
For the remaining term, by Poincar\'e inequality we have
\begin{eqnarray}
&&\left<\pa_x^\al(\vp-\bar{\vp}),\dive(\rho_i\pa_x^\al(\psi_i-\psi_e))\right>\nonumber\\
&=&\left<\pa_x^\al(\vp-\bar{\vp}),\D\rho_i\pa_x^\al(\psi_i-\psi_e)\right>+\left<\pa_x^\al(\vp-\bar{\vp}),\rho_i\dive \pa_x^\al(\psi_i-\psi_e)\right>,\nonumber
\end{eqnarray}
in which
\[
\left|\left<\pa_x^\al(\vp-\bar{\vp}),\D\rho_i\pa_x^\al(\psi_i-\psi_e)\right>\right|\leq C\|\pa_x^\al\calf\|\|\D \rho_i\|_\infty\|\pa_x^\al(\Ni-\Ne-\eps^2\bar{\rho}_i^1)\|\leq C\delta\mathcal{D}(t)+C\eps^2\|\bar{\rho}_i^1\|_{s-1}^2,
\]
and
\begin{eqnarray*}
    &&\left<\pa_x^\al(\vp-\bar{\vp}),\rho_i\dive \pa_x^\al(\psi_i-\psi_e)\right>\nonumber\\
   & = & -\dfrac{1}{2}\|\pa_x^\al \calf\|^2 -\left<\pa_x^\al(\vp-\bar{\vp}), \D\rho_i \pa_x^\al\calf \right>- \eps^2\left\langle \partial^\al_x(\varphi-\bar{\varphi}), \rho_i\partial^\al_x \bar{\rho}_i^1\right\rangle\\
    &\leq &  -\dfrac{1}{2}\|\pa_x^\al \calf\|^2 + C\eps^2 \mathcal{D}(t)+C\delta\mathcal{D}(t) + C\eps^2 \|\bar{\rho}_i^1\|_{s-1}^2.
\end{eqnarray*}
Combining all these estimates and integrating the first equation over $[0,t]$ we conclude \eqref{Fintegral}.
\end{proof}

\subsection*{Proof of Theorem \ref{thm4.1}} Combining \eqref{Nifinal}, \eqref{Nefinal}, \eqref{Fintegral} and taking the superior with respect to $t$, we have
\begin{eqnarray*}
&&\int_0^t\|(\Ni,\Ne,\calf)(\tau)\|_{s-1}^2d\tau +\dfrac{1}{\eps^2}\|\psi_i(t)\|_{s-1}^2+\dfrac{1}{\eps^2}\sup_{0\leq \tau \leq t}\|\psi_i(\tau)\|_{s-1}^2\\
&\leq& C\eps^2+C\delta\int_0^t\mathcal{D}(\tau)d\tau+\|w_e(t)\|_{s-1}^2+\int_0^t\|w_e(\tau)\|_{s-1}^2d\tau +\dfrac{C \delta+c_0\mu_0}{\eps^2} \sup_{0\leq \tau \leq t}\|\psi_i(\tau)\|_{s-1}^2.
\end{eqnarray*}

Imposing $\mu_0$ sufficiently small satisfying
\begin{equation}\label{mu0}
c_0 \mu_0<\dfrac{1}{2}
\end{equation}
leads to
\begin{eqnarray*}
&&\int_0^t\|(\Ni,\Ne,\calf)(\tau)\|_{s-1}^2d\tau +\dfrac{1}{\eps^2}\|\psi_i(t)\|_{s-1}^2+\dfrac{1}{\eps^2}\sup_{0\leq \tau \leq t}\|\psi_i(\tau)\|_{s-1}^2\\
&\leq& C\eps^2+C\delta\int_0^t\mathcal{D}(\tau)d\tau+\|w_e(t)\|_{s-1}^2+\int_0^t\|w_e(\tau)\|_{s-1}^2d\tau +\dfrac{C \delta}{\eps^2} \sup_{0\leq \tau \leq t}\|\psi_i(\tau)\|_{s-1}^2.
\end{eqnarray*}

Substitute the above estimate into \eqref{regue} implies
\begin{eqnarray*}
 &&\|(\Ne,w_e,\mathcal{F})(t)\|_{s-1}^2 + \int_0^t\mathcal{D}(\tau)d\tau+\dfrac{1}{\eps^2}\sup_{0\leq \tau \leq t}\|\psi_i(\tau)\|_{s-1}^2 \\
 &\leq&  C\eps^2+C\delta\int_0^t\mathcal{D}(\tau)d\tau+\dfrac{C \delta}{\eps^2} \sup_{0\leq \tau \leq t}\|\psi_i(\tau)\|_{s-1}^2.
\end{eqnarray*}

By choosing $\delta>0$ sufficiently small and combining \eqref{finalpre2} and \eqref{regui} ends the proof. \hfill $\square$

\subsection*{Acknowledgements} The research of this work was supported in part by the National Natural Science Foundation of China [grant numbers 11831011,12161141004].This work is also partially
supported by Institute of Modern Analysis–A Shanghai Frontier Research Center.

	\bibliographystyle{abbrv}
	\bibliography{convra-2to1-EP}

\end{document}